\documentclass[12pt,fleqn]{article}
\usepackage{amssymb,amsmath}
\usepackage{mathtools,mathrsfs,bm}

\newcommand{\be}{\begin{equation}}
\newcommand{\ee}{\end{equation}}

\newcommand{\wt}{\widetilde}
\newcommand{\wh}{\widehat}
\newcommand{\de}{\delta}

\newcommand{\al}{\alpha}
\newcommand{\bt}{\beta}

\newcommand{\si}{\sigma}

\newcommand{\lb}{\lambda}

\renewcommand{\th}{\theta}

\newcommand{\ep}{\varepsilon }

\newcommand{\Z}{\mathbb{Z}}
\newcommand{\C}{\mathbb{C}}
\newcommand{\Id}{{I}}
\newcommand{\Var}{\operatorname{Var}}

\newcommand{\Tr}{\operatorname{Tr}}

    \def\P2n{{\rm P}_{{\rm II}}^{(n)}}
\def\T {\mathbb{T}}

    \newtheorem{theorem}{Theorem}[section]
    \newtheorem{lemma}[theorem]{Lemma}
    
    \newtheorem{proposition}[theorem]{Proposition}
    
    \newtheorem{Remark}[theorem]{Remark}
    \newenvironment{remark}{\begin{Remark}\rm}{\end{Remark}}
    \newtheorem{Example}[theorem]{Example}
    
    \newtheorem{Assumptions}[theorem]{Assumptions}

    \newenvironment{proof}
    {\rm \trivlist \item[\hskip \labelsep{\it Proof. }]}
    {\hspace*{\fill}$\Box$\endtrivlist}

    \newcommand{\supp}{{\operatorname{supp}\;}}
    \newcommand{\sgn}{{\operatorname{sgn}}}

\newcommand{\bbz}{\mathbb{Z}}
\newcommand{\bbr}{\mathbb{R}}
\newcommand{\bbc}{\mathbb{C}}

\newcommand{\la}{\label}
\newcommand{\ci}{\cite}
\newcommand{\bi}{\bibitem}
\newcommand{\meas}{\mbox{meas}}
\newcommand{\dimH}{\mathrm{dim_H}}

\begin{document}
\bigskip\bigskip\bigskip
\begin{center}
{\Large\bf Critical almost Mathieu operator:
hidden singularity, gap continuity, and the Hausdorff dimension of the spectrum}\\
\bigskip\bigskip
S. Becker,\footnote{Bocconi University, Milan, Italy}
S. Jitomirskaya,\footnote{UC Berkeley, USA and UC Irvine, USA}
I. Krasovsky\footnote{Imperial College London, UK, and SISSA, Trieste, Italy}\\
\bigskip \bigskip\bigskip\bigskip
{\it In memory of David Thouless, 1934--2019}
\end{center}

\bigskip

\begin{abstract}
\noindent
We obtain a representation of the critical
almost Mathieu family as a Jacobi matrix that has a singularity. This allows us to
prove  that the Hausdorff dimension of its spectrum is not larger than $1/2$ for all irrational frequencies, solving a long-standing problem.
Other corollaries include two very short proofs of zero measure of
the spectrum (e.g., Problem 5 in B. Simon's list of the 21st century problems). We also obtain continuity of the measure of the spectrum for general singular Jacobi matrices,
{ and prove a similar Hausdorff dimension result for the quantum graph graphene.}
\end{abstract}

\section{Introduction}
The critical almost Mathieu operator on $\ell^2(\mathbb{Z})$, that is
\be\label{AM}
(H_{\al,\th,\lb}\phi)(n)=\phi(n-1)+\phi(n+1)+2\lambda\cos 2\pi(\al n+\th)\phi(n),
\qquad n=\dots,-1,0,1,\dots
\ee
with $\lambda=1,$ represents a transition from ``extended states''
for the subcritical case, $0\le\lambda<1,$ to ``localized states'' for the
supercritical case, $\lambda>1.$ The almost Mathieu family
models an electron on the two-dimensional square lattice in a uniform
perpendicular magnetic field with flux $2\pi\alpha$ \cite{pei}; the critical case
corresponds to the isotropic lattice, while $\lambda \not= 1$
corresponds to anisotropy. The almost Mathieu operator
$H_{\al,\th}:= H_{\al,\th,1}$ at the critical value of the parameter $\lambda$ is the
one most important in physics, where it is also known as the Harper
or Azbel-Hofstadter model (see, e.g., \cite{av} and references therein), and also the one least understood mathematically
and even heuristically. In our approach below, we use the fact that the problem is two-dimensional in that we rely on an appropriate
change of the gauge of the magnetic field.

The spectrum of $H_{\al,\th} $ for
irrational $\alpha$ is a $\theta$-independent\footnote{Also for any $\lb\neq 0$.}
fractal, beautifully depicted via the Hofstadter butterfly
\cite{hof}. There have been
many numerical and heuristic studies of its fractal dimension in
physics literature (e.g., \cite{ketz,gei,tk,wa}).
A conjecture, sometimes
attributed to Thouless (e.g., \cite{wa}), and appearing
already in the early 1980's, is that the dimension is
equal to 1/2.  In 1994 Last proved \cite{L} that the Hausdorff dimension
is not larger than 1/2 for a measure zero set of $\al$.
The original conjecture has been rethought after rigorous and numerical studies demonstrated that the Hausdorff
dimension can be less than 1/2 (and even be zero) for some $\alpha$
\cite{ls, alsz, wa}, while packing/box counting dimension can be higher (even equal to one)
for some (in fact, of the same!) $\alpha$ \cite{jz}. Namely, zero Hausdorff dimension for a
subset of Liouville $\al$ was proved by Last and Shamis \cite{ls}, a
result recently extended to all
weakly Liouville\footnote{ We say $\al$ is weakly Liouville if
  {$\beta(\al):=\limsup_{n\to\infty}\frac{-\ln \|n\al\|_\T}{n}>0$}, where
  $\|\th\|_\T=\mbox{dist}\,(\th, \mathbb{Z}).$}  $\al$  by Avila, Last,
Shamis, Zhou \cite{alsz}. At the same time, for
weakly Liouville $\al,$ the  packing (and therefore box counting)
dimension is equal to one \cite{jz}!  Helffer, Liu, Qu, and
Zhou \cite{hlqz,hlqz1} proved the existence of a dense
positive Hausdorff dimension set of Diophantine $\al$ with  positive
Hausdorff dimension\footnote{{ that can be made arbitrarily small \cite{hlqz1}}} of the { spectrum.}

All those results, as
well as heuristics by Wilkinson and Austin \cite{wa} and, of course, numerics, hold for
measure zero sets of $\alpha,$ { and
therefore leave the almost-everywhere dimension open.} There is also a
conjecture attributed to J. Bellissard  (e.g., \cite{ls,hlqz}) that the
dimension of the spectrum is a
property that only depends on the tail in the continued fraction
expansion of $\alpha$ and thus should be the same for a.e. $\alpha$
(by the properties of the Gauss map). Recently, B. Simon listed the problem to
determine the Hausdorff dimension of the spectrum of the critical almost
Mathieu on his new list of hard unsolved problems \cite{linde}.

The equality in the original Thouless
conjecture can be viewed as  two inequalities, and here we prove one of
those for {\it all} irrational $\alpha.$ This is also the first result
on the fractal dimension that holds for more than a measure zero set of $\alpha$. Denote the spectrum of an operator $K$ by $\si(K)$,  the Lebesgue
measure of a set $A$ by $|A|$, and its Hausdorff dimension by $\dimH(A).$
We have

\begin{theorem}\label{Hdimthm}
For any irrational $\al$ and real $\th$,  $\dimH(\si(H_{\al,\th}))\leq1/2.$
\end{theorem}

Of course, it only makes sense to discuss upper bounds on the Hausdorff dimension of a set on the real line once its Lebesgue measure is shown to be zero. The Aubry-Andre conjecture stated that the measure of the spectrum of
$H_{\al,\th,\lb}$ is equal to $4|1-|\lb||,$ so to $0$ if $\lb=1,$ for any irrational $\alpha.$
This conjecture was popularized by B. Simon, first in his list of 15 problems in mathematical physics \cite{15} and then, after it was proved by Last for a.e. $\alpha$ \cite{l1,L}, again as Problem 5 in \cite{XXI}, which was to prove this conjecture for the remaining measure zero set of $\alpha$, namely, for $\al$ of bounded type.\footnote{An irrational number $\al$ is said to be of bounded type if the sequence of coefficients of the continued fraction expansion of $\al$ is bounded.} The arguments of \cite{l1,L} did not work for this set, and even though the semi-classical analysis of Helffer and Sj\"ostrand \cite{hs} applied to some of this set for $H_{\al,\th}$, it did not apply to other such $\al$, including, most notably, the golden mean --- the subject of most numerical investigations. For the non-critical case, the proof for all $\al$ of bounded type  was given in \cite{jk}, but the critical ``bounded-type'' case remained difficult to crack. This remaining problem for zero measure of the spectrum of $H_{\al,\th}$ was finally solved by Avila and Krikorian \cite{ak}, who employed a deep dynamical argument. The argument of \cite{ak} worked not for all $\alpha,$ but for a full measure subset of Diophantine $\alpha$.
A proof that works simultaneously for all irrational $\alpha$ follows from the renormalization as in \cite{ak} combined with all-frequency reducibility to rotations as in \cite{afk}; for example, it can be extracted from a related argument in \cite{ak2} when combined with the results of \cite{jb1,bj}. Another proof follows from the results of \cite{global}. All those arguments rely on highly non-trivial work.
In the present paper, we provide two very simple proofs that recover this theorem (also simultaneously for all irrational $\alpha$) and thus give { a very simple}
solution to Problem 5 of \cite{XXI}.

\begin{theorem}\label{zero}
For any irrational $\al$ and real $\th$,  $|\si(H_{\al,\th})|=0.$
\end{theorem}

The key to our proofs is the chiral gauge transform. In Section \ref{chiral}, we introduce this
{transform and show that the direct integral in $\th$ of operators
$H_{2\al,\th}$ is isospectral with the direct integral in $\th$ of $\wh H_{\al,\th}$} given by
\be\label{chi}
 (\wh H_{\al,\th}\phi)(n)=2\sin 2\pi(\al (n-1) +\th)\phi(n-1)+
 2\sin 2\pi(\al n +\th)\phi(n+1).
\ee
This representation of the almost Mathieu operator corresponds to choosing the chiral gauge for the
perpendicular magnetic field applied to the electron on the square
lattice. This gauge was previously discussed
non-rigorously in \cite{mz,KH,WZnpb} and corresponds, in a certain sense, to a 45 degree
rotation of the lattice. 
In the rational case $\al=p/q$, a similar representation for the discriminant
of $H_{p/q,\th}$ is easy to justify (see, e.g., the appendix of
\cite{kprb}) and was used in \cite{K} to obtain estimates on the size of central bands of $H_{p/q,\th}$.
The chiral gauge
transform developed in the present work was recently used to obtain some further new results.
In \cite{j} it was used to construct a new
Fourier-type transform that led to a solution of another long-standing
problem: absence of eigenvalues for $H_{\al,\th}$ for all $\al$, $\th$.
In \cite{jkk} it was used to improve the upper bound for the measure of the spectrum in the rational $\al$ case. A finite-dimensional non self-adjoint deformation of the chiral gauge has recently also been used to study the spectra of non self-adjoint deformations of the almost Mathieu operator \cite{BO}. 
For the intuition behind this transform as well as a pictorial description we refer
the reader to \cite{cdm}.

The advantage of (\ref{chi}) is that for irrational $\al$ it is a {\it
  singular} Jacobi matrix, { i.e., one with} off-diagonal elements
not bounded away from zero, so
that the matrix quasi-separates into blocks. This is already sufficient to conclude Theorem \ref{zero}, based on
{ Kotani theory} which we do
in Section \ref{proofzero}. { Yet another very simple proof of Theorem \ref{zero}, which avoids Kotani theory, follows from Theorem \ref{measure} below.}

Consider a general class of quasiperiodic $C^1$ Jacobi matrices, { that is, operators} on
$\ell^2(\mathbb Z)$  given by
\be\label{H2}
(H_{v,b,\alpha,\theta}\phi)(n)=b(\th + (n-1)\al)\phi(n-1)+b(\th+n\al)
\phi(n+1)+v(\th+n\al)\phi(n),
\ee
with  { $b,v\in C^1(\mathbb R;\mathbb R)$, both periodic with period $1$}. We further assume that the number of zeros of $b(x)$ on a period, if any, is at most finite.

{Let $M_{v,b,\al}$ be the direct integral of $H_{v,b,\al,\th}$ over $\th\in\T,$
\be\label{M}
M_{v,b,\al}=\int_{\T}^{\oplus}H_{v,b,\al,\th}\,d\th.
\ee}
For the critical almost Mathieu family, we abbreviate
\[
M_\gamma:=M_{2\cos(2\pi\cdot),1,\gamma}.
\]

Continuity in $\al$ of $\si(M_{v,b,\al})$ in the Hausdorff metric was proved in \cite{AS1983,El82}.
{ Continuity } of the measure of the spectrum is  a more delicate issue, since,
in particular, $|\si(M_{\al})|$ can be (and is, for the almost Mathieu operator)  discontinuous at
rational $\al$. { An approach going back to \cite{ams,l1,L} derives the measure of the spectrum for irrational $\alpha$ from those of its rational approximants
by combining quantitative Hausdorff continuity of gaps
with a bound on the number of intervals in the spectrum of rational approximants.}

In the Schr\"odinger case, that is for $b=1$, Avron, van Mouche, and Simon
\cite{ams} obtained a very general result on  H\"older-$\frac 12$ continuity { of gaps}
(for arbitrary
$v \in C^1$), improving H\"older-$\frac{1}{3}$
continuity obtained earlier by Choi, Elliott, and Yui
\cite{CEY}.

{ For frequencies with unbounded continued-fraction coefficients,
Last \cite{l1,L} showed convergence of the measure of the spectrum on subsequences of canonical periodic approximantes
yielding the noncritical measure formula in \cite{l1} and
the critical case in \cite{L}.
The use of a favorable subsequence of approximants,
necessitated by the $1/2$-H\"older continuity of spectra modulus, persisted in later
periodic-approximation results, e.g. \cite{jmarx,jmavi} for the same reason.\footnote{
\cite{jmarx} also gives a full-sequence improvement for analytic potentials.}   It was argued, however, already} in \cite{ams}
that  H\"older continuity of {\it any} order larger than $1/2$ would imply
the desired continuity property of the measure of the
spectrum for {\it all} { irrational} $\al$.

{ Here we provide an extension of the  $b\equiv 1$ argument of \cite{ams} to general $b$ that leads to a desired modulus improvement for {\it singular} $b.$ }

A Jacobi matrix (\ref{H2}) is called {\it singular} if for some $\th_0$,
$b(\th_0)=0.$
  Theorem \ref{continuitylemma1} establishes { the desired quantitative continuity improvement}
and allows us to conclude { the approximation property of} the measure of the spectrum for singular Jacobi matrices { for any irrational $\al$ and the entire sequence of canonical approximants}:

\begin{theorem}\label{measure}
For singular  $H_{v,b,\al,\th}$
as above, for
any irrational $\al$  and its sequence of canonical
approximants $\frac{p_{n}}{q_{n}}$ we have
 \be \label{si}
|\si(M_{v,b,\al})|=\lim_{n\to\infty} \left|\si\left(M_{v,b,\frac{p_{n}}{q_{n}}}\right)\right|.\ee
\end{theorem}

\begin{remark} \begin{enumerate} \item The measure of the
  spectrum can often be estimated for an arbitrary
rational; see \cite{bhj} and \cite{jkk}. \item { {For Schr\"odinger operators with analytic potentials and
Diophantine frequencies, almost Lipschitz spectral-approximation
estimates were obtained in \cite{jk} under the assumption of positive
Lyapunov exponents, uniformly in energy outside sets of
arbitrarily small Lebesgue measure,
which completed the non-critical case of { the Aubry--Andr\'e measure formula}. See also Bellissard \cite{bel} for earlier Lipschitz estimates for a large class of quasiperiodic operators,
however without a uniform  constant, { thus not allowing one to conclude}
continuity of the measure of the spectrum.}}
\end{enumerate}
  \end{remark}

{ For Schr\"odinger operators ($b=1$), convergence of the measure
of the spectrum along (suitable) sequences of rational approximants
 was previously established  under various conditions} in
\cite{jk,jmavi,jmarx,zhao}.

Typically, proofs that work for
$b=1$ extend also to the case of non-vanishing $b$, i.e., {\it non-singular} Jacobi
matrices, and there is no reason to believe the results of
\cite{jmarx,zhao} should be an exception. On the other hand, extending various Schr\"odinger results to the singular
Jacobi case is usually technically non-trivial and adds a significant degree
of complexity (e.g. \cite{jmarxreview,ajm}). Here however, we
show that a singularity can be {\it exploited},
rather than circumvented, { leading, moreover, to a strictly stronger result.}

 { { Namely, we}  extend the argument of \cite{ams} to  $b(x)\neq 1$ by constructing}  test function cutoffs whose increments are inversely proportional to the hopping coefficients.
   { This gives the one-sided Hölder-$\frac{1}{2}$ spectral-approximation estimate for \eqref{H2} for general $b(x)$ and every irrational $\al$, with an improved modulus for singular $b$ leading to convergence of the Lebesgue measure of the spectrum along the entire sequence of canonical approximants.}

{ Our argument for the Hausdorff-dimension bound is also based on singularity, but in a different way. Namely,} we construct spectral covers directly. Continued-fraction return times divide the singular Jacobi matrix into blocks of lengths $q_n$ and $q_{n-1}$, joined by bonds of size of order $q_n^{-1}$. A boundary-resolvent argument reduces the spectral inclusion to small boundary perturbations of these blocks. We then control their phase dependence by a commutator identity whose error has uniformly bounded trace norm. This yields a cover by at most $q_n+q_{n-1}$ intervals of total length of order
$q_n^{-1}$.

In the case of the almost Mathieu operator, the original intuition behind Thouless' conjecture  on the
Hausdorff dimension 1/2 is based on { another fascinating conjecture of Thouless} \cite{T83,Tcmp}: that
{for the critical almost Mathieu operator
$M_{\gamma}=\int_{\T}^{\oplus}H_{\gamma,\th}\,d\th$}, for any irrational $\al$,
in the limit $p_n/q_n\to \al$, we have
$q_n|\si(M_{p_n/q_n})|\to c$ where $c=32 G/\pi,$ $G$ being  the Catalan constant.
Thouless argued that if $\si(M_{\al})$ is ``economically covered'' by $\si(M_{p_n/q_n})$ and
if all bands are of about the same size,\footnote{In reality, the
  bands can decay exponentially with distance from the
  center at each step in the continued fraction hierarchy
  \cite{hs}. The central bands can be power-law small in $1/q$
  \cite{K}. However, ``economically covered''  is a physicist's way of
  stating a { certain} modulus of continuity, and Thouless' intuition does work for the upper bound.}
  then { such a cover suggests the Hausdorff-dimension upper bound $1/2$.} Clearly, the exact value of $c>0$ is not important for
this argument. An upper bound of the form \be\label{bound}
q_n|\si(M_{p_n/q_n})| { \le} C,\qquad n=1,2,\dots,
\ee
was proved by Last \cite{L} (with $C=8e$, which was recently improved in \ci{jkk} to
$C=4\pi$), which, combined with
H\"older-$\frac{1}{2}$ continuity, led him in \cite{L} to the proof of zero-measure spectrum for all irrational $\al$ except those of bounded type and
to the bound $\le \frac{1}{2}$ for the
Hausdorff dimension for  irrational $\alpha$ (of zero measure)
satisfying $\lim_{n\to\infty}|\al-p_n/q_n|q_n^4=0$.

Now the chiral gauge representation, the bound (\ref{bound}), and Theorem \ref{measure} immediately give
a proof of Theorem \ref{zero} for all irrational $\al$.

However, this slight improvement of H\"older-$\frac{1}{2}$ continuity is insufficient to prove Theorem \ref{Hdimthm}.
Instead, we construct appropriate covers for the spectrum using the singular matrix structure of \eqref{chi} directly, namely the fact that
it quasi-separates into blocks. Consider the sequence $p_n/q_n$ of canonical approximants to an irrational $\alpha$.
Then the off-diagonal
elements of \eqref{chi} become of order $1/q_n$ at distances $q_n$ or $q_{n-1}$, as shown by a simple continued fraction argument (see Lemma \ref{lemma-rt}). Thus we only obtain blocks of two sizes, $q_n$ and $q_{n-1}$ connected by a perturbation of norm of order $1/q_n$.
We will cover the spectrum by small intervals around eigenvalues of the separated blocks.

 { Simply removing the connecting bonds gives an operator-norm error of order $q_n^{-1}.$ Applying this estimate separately to the block eigenvalues would produce a cover whose total length is only bounded by a constant. We instead retain the boundary structure of the perturbation and use trace-norm estimates to obtain {\it total length} of order $q_n^{-1}.$}
This is done by showing in Lemma \ref{lemma-subm}  that the original spectrum is a subset
of the closed union of spectra of slightly (finite rank, small of order $1/q_n$ norm) perturbed blocks. This crucial Lemma
is closely related to the standard
pseudospectral characterization of small perturbations \cite[Theorem~2.2.3]{Sjostrand19}. (For a closed operator
$P$, the $\varepsilon$-pseudospectrum is the union of the spectra of $P+Q$
over all bounded perturbations $Q$ with $\|Q\|<\varepsilon$.)

Now if the difference between blocks of the same size $N$ ($N=q_n$ or $N=q_{n-1}$) were only fixed (independent of $N$) finite rank, norm of order $1/q_n$, we would use Lidskii inequalities to show { that the spectrum} is covered by $q_n+q_{n-1}$ intervals of small total length around the eigenvalues of some two of them.
However, a difference between blocks also occurs because of small changes in the phase of all the matrix elements from block to block.

To deal with that, we use an argument based on unitary evolution,
Lemma \ref{lemma-flow}: let $H(x)$ be an operator depending on a parameter $x$ in a smooth way and suppose there is a constant
{ skew-Hermitian operator $K$, $K^*=-K$, such that $e^{-xK}$ is unitary and}
$\frac{d}{dx}H(x)=[K,H(x)]\equiv KH(x)-H(x)K$.
In the classical inverse scattering method in integrable models one uses the fact that
the evolution of this form preserves the spectrum of $H$ (isospectral deformation). Indeed, under this condition the unitarily equivalent operator $e^{-xK}H(x)e^{xK}$ has zero derivative. The operators $K$ and $H$ appearing in the commutator are called a Lax pair.

{ It turns out that there exists} a simple matrix $K$ (see Lemma \ref{lemma-commut} below)
such that the Lax pair representation applies to
a block of the almost Mathieu operator $H$ up to a  {\it rank 4}  error. (This fact is related to the duality transform.) The error is produced by the boundary conditions as the block is finite-dimensional.
Since the phase in our case ranges only over a small interval, the finite-rank property of the error
allows us to conclude that the variation of the eigenvalues is small by applying a Lidskii inequality.
This in turn allows us to finish the proof that the spectrum is covered by at most $q_n+q_{n-1}$ intervals of total length of order $1/q_n$.
The proof of Theorem \ref{Hdimthm} now easily follows.

{
The argument above also applies to the model of graphene introduced in \cite{bhj}, for which an analogous commutator estimate holds.
We present the corresponding result in { Theorem \ref{gr}} { in the Appendix.}
{ In the forthcoming paper \cite
{BJK2} we extend this approach to the full singularizable critical extended Harper class.
}

The rest of the paper is organized in the following way. We
introduce the chiral gauge transform in Section \ref{chiral} and use it to give a
short proof of Theorem \ref{zero} in Section \ref{proofzero}.
 In Section \ref{sectioncontinuity} we
show the continuity of gaps, Theorem \ref{continuitylemma1}, which includes an improved modulus of continuity for
singular Jacobi matrices. We then use it
 to prove Theorem \ref{measure} and give another proof of Theorem \ref{zero} in Section \ref{sectionproofmeasure}.
In Section \ref{section-cover} we construct small-measure finite covers for the spectrum of the critical almost Mathieu operator. Section \ref{section-hausdorff} contains the proof of Theorem \ref{Hdimthm}. In the Appendix, we discuss Theorem  \ref{gr}.
}

{\it Statement on the use of artificial intelligence:} ChatGPT 5.6 Sol was used to accelerate the determination of
the explicit matrix expressions appearing in Lemma~7.6. ChatGPT 5.6 Sol was also used in the determination of the corresponding expressions in the Appendix and
assisted in preparing  preliminary revisions of Section~7 and of the Appendix, which were
subsequently fully rewritten by the authors.
GPT 6 was used at a later stage for additional proofreading and linguistic corrections. 

\section{Preliminaries}
\subsection{Hausdorff dimension}
Denote the $t$-dimensional Hausdorff measure of a set $A$ by
\[
\meas_t(A)= \lim_{\de\downarrow 0}\inf\left\{\sum_{m=1}^{\infty}|w_m|^t: \cup_{m=1}^{\infty} w_m \mbox{ is a $\de$-cover of $A$}\right\},
\]
where the infimum is taken over all $\de$-covers of $A$ by intervals: $A\subset\cup_{m=1}^{\infty} w_m$, where $w_m$ is an interval with $|w_m|\le\de$.
The Hausdorff dimension of $A$ is then
\[
\dimH(A)=\inf\{t>0: \meas_t(A)<\infty\}.
\]

\subsection{IDS  and the Lyapunov exponent}

For a
family of operators $H_\th$ on $\ell^2(\mathbb Z)$
such that $T^{-1}H_\th T = H_{\th-\al}$, where $(T\phi)(n)=\phi(n+1)$, $\forall n$,
let $\nu_{\theta}$ be the spectral measure of $H_{\theta}$ corresponding to $e_0=(\dots,0,e_0(0)=1,0,\dots)$. { Namely, for any Borel set} $A$ we have
\begin{align*}
\nu_{\theta}(A)=( \chi_A(H_{\th})e_0,e_0),
\end{align*}
where $\chi_A(x)$ is the indicator function of $A$.
The {\it density of states measure} $\rho$ is defined by
\begin{align*}
\rho(A)=\int_{\T}\nu_{\theta}(A)\ \mathrm{d}\theta, \qquad \T=\mathbb{R}/\mathbb{Z}.
\end{align*}
We have \cite{AS1983}
\[
\si\!\left(\int_{\T}^{\oplus}H_\th\,d\th\right)=\supp\,\rho.
\]
The cumulative distribution function $N(E):=\rho((-\infty, E))$ of { $\rho$} is called the {\it integrated density of states
  (IDS)} of $H_{\th}$.

For a Jacobi matrix (\ref{H2}) we label the zeros of $b(\theta)$ on the period, whose number we assume to be finite, by $\theta_1, \theta_2,\dots,\theta_m$. (For $b(\theta)=2\sin 2\pi\th$ from (\ref{chi}), we have
  two zeros $\th_1=0$, $\th_2=1/2$.)
Let $\Theta=\cup_{j=1}^m \cup_{k\in \bbz}\left\lbrace
  \theta_j+k\al\right\rbrace$.

For $\theta\notin \Theta$, the eigenvalue equation $H_{v,b,\al,\th}\phi=E \phi$ has the following dynamical reformulation:
\begin{align*}
\left(
\begin{matrix}
\phi(n+1)\\
\phi(n)
\end{matrix}
\right)=
A^{E}\left(\theta+n\al \right)
\left(
\begin{matrix}
\phi(n)\\
\phi(n-1)
\end{matrix}
\right),
\end{align*}
where
\begin{align*}
\mathrm{GL}(2, \bbc)\ni A^{E}(\theta)=
\frac{1}{b(\theta)}
\left(
\begin{matrix}
E-v(\theta)\ &-b(\theta-\al)\\
b(\theta)\ &0
\end{matrix}
\right)
\end{align*}
is called the {\it transfer matrix}.
Let
\begin{align*}
A^{E}_{n}(\theta)=A^{E}(\theta+(n-1)\al)\cdots A^{E}(\theta+\al)A^{E}(\theta)
\end{align*}
be the {\it n-step transfer matrix}.

{ When $\log|b|\in L^1(\T)$, the} {\it Lyapunov exponent} of $H_{v,b,\al,\theta}$ at energy $E$ is defined as
\begin{align}\label{defLE}
L(E)=\lim_{n\rightarrow\infty}\frac{1}{n}\int_{\T}\ln{\|A^E_n(\theta)\|}\ \mathrm{d}\theta.
\end{align}
{Whenever $\log|b|\in L^1(\T)$, the Thouless formula
(e.g., \cite{tes}) links $N(E)$ and $L(E)$:}

\be \label{thou}L(E) = -\int_{\T} \ln |b(\th)|d\th + \int_{\mathbb{R}} \ln|E-E_1|dN(E_1).\ee

Note that both for $b(\th)=1$ and $b(\theta)=2\sin 2\pi\th$, we have
\be\label{thou2}
\int_{\T} \ln |b(\th)|d\th =0.
\ee

\subsection{Continued fraction expansion}
Let $\alpha\in (0,1)$ be irrational. Then $\alpha$ has the following continued fraction expansion
\begin{align*}
\alpha=\frac{1}{a_1+\frac{1}{a_2+\frac{1}{a_3+\cdots}}},
\end{align*}
with $a_n$ positive integers for $n\geq 1$.

The reduced rational numbers
\begin{align}\label{defpnqn}
\frac{p_n}{q_n}=\frac{1}{a_1+\frac{1}{a_2+\frac{1}{\cdots+\frac{1}{a_n}}}},\qquad n=1,2,\dots,
\end{align}
are called the {\it canonical approximants} of $\alpha$.
They satisfy the recursion $p_n=a_np_{n-1}+p_{n-2}$, $q_n=a_nq_{n-1}+q_{n-2}$, $n=2,3,\dots$, $p_0=0$, $q_0=1$, from which it follows that
\be\label{detid}
q_{n}p_{n-1}-p_n q_{n-1}=(-1)^n,\qquad n=1,2,\dots.
\ee
This in turn implies that $p_n$, $q_n$ are coprime and
\be\label{nnm1}
\frac{p_{n-1}}{q_{n-1}}-\frac{p_n}{q_n}=\frac{(-1)^n}{q_nq_{n-1}},\qquad n=1,2,\dots.
\ee
It is well known that
\be\label{pdord}
\frac{p_{2k}}{q_{2k}}<\alpha<\frac{p_{2k-1}}{q_{2k-1}},\qquad k=1,2,\dots,
\ee
the sequence $\frac{p_{2k-1}}{q_{2k-1}}$ is strictly decreasing while $\frac{p_{2k}}{q_{2k}}$ is strictly increasing,
$q_n\ge 2^{(n-1)/2}$, and
\begin{align}\label{q}
\frac{1}{q_n(q_n+q_{n+1})}<\left|\alpha-\frac{p_n}{q_n}\right|<\frac{1}{q_{n}q_{n+1}}<\frac{1}{q_{n}^2},\qquad n=1,2,\dots
\end{align}

\section{Chiral gauge}\label{chiral}
Let $\al,\th\in\bbr$ and consider the following operator on $\ell^2(\mathbb Z)$:
\be
\wt H_{\al,\th}=\wh H_{\al,1/4+\al/2+\th},
\ee
in terms of the operator defined in (\ref{chi}).

{ Define the following unitary operators on $L^2(\T;\ell^2(\mathbb Z))$:}

\begin{align}
(T\phi)(n,\th)&=\phi(n+1,\th),\qquad
(S\phi)(n,\th)=e^{2\pi  i(\th+n\al)}\phi(n,\th),\la{SandT}\\
(U_x\phi)(n,\th)&=e^{2\pi i n x (\th+n\al/2)}\phi(n,\th),\la{Q}\\
(R\phi)(n,\th)&=\sum_{k\in\mathbb{Z}}e^{-2\pi i
  k(\th+n\al)}\int_{\T}e^{-2\pi i n\bt}\phi(k,\bt)d\bt.\la{R}
\end{align}

Here and below, phases $\theta\in\T$ are represented by their unique representatives in $[0,1)$. For $x\in\mathbb R$, we also write
\[
(S^x\phi)(n,\theta)=e^{2\pi i x(\theta+n\alpha)}\phi(n,\theta).
\]

{Note that the decomposable operators with fibers $H_{\al,\th}$ and
$\wt H_{\al,\th}$ have the following { fiberwise} representations in terms of
$S$, $T$ (while $R$ is not decomposable):}
\begin{align}
H_{\al,\th}&=T+T^{-1}+S+S^{-1},\\
\wt
H_{\al,\th}&=e^{i\pi\al}(ST+S^{-1}T^{-1})+e^{-i\pi\al}(ST^{-1}+S^{-1}T).
\end{align}

{Consider the direct integrals (particular cases of (\ref{M}))
\be\la{wtM}
M_{\al}=\int_{\T}^{\oplus} H_{\al,\th}\,d\th,\qquad
\wt M_{\al}=\int_{\T}^{\oplus}\wt H_{\al,\th}\,d\th.
\ee}

{Now note that on $L^2(\T;\ell^2(\mathbb Z))$
\be
T^2+T^{-2}+S+S^{-1}
=\int_{\T}^{\oplus}
\left(H^{(1)}_{2\al,\th}\oplus H^{(2)}_{2\al,\th+\al}\right)\,d\th,
\ee}
 where $H^{(1)}_{\al,\th}$ operates on $\ell^2(2\mathbb Z)$, and
 $H^{(2)}_{\al,\th}$  on $\ell^2(2\mathbb Z+1)$
in the sense that
 $H^{(j)}_{\al,\th}=P_j^{-1}H_{\al,\th}P_j$, where
 $P_1:\ell^2(2\mathbb Z)\to \ell^2(\mathbb Z)$ and $P_2:\ell^2(2\mathbb
 Z+1)\to \ell^2(\mathbb Z)$ are the natural unitaries defined by
 $(P_1u)(n)=u(2n)$ and  $(P_2u)(n)=u(2n+1)$.
 Therefore, with the definition
 \[{M^{(j)}_{\al}=\int_{\T}^{\oplus}H^{(j)}_{\al,\th}\,d\th,\qquad j=1,2,}\]
 the operator
 $T^2+T^{-2}+S+S^{-1}$ is unitarily equivalent to $M^{(1)}_{2\al}\oplus M^{(2)}_{2\al}$.
 We see that $M^{(j)}_{2\al}$, $j=1,2$, are two copies of the operator $M_{2\al}$ represented
on the subspaces corresponding to the even and the odd indices $n$ of $\phi(n,\theta)$, respectively.

We have

\begin{theorem}\label{chiralrepresentthm} Let $\al\in\bbr$.
The operators $T^2+T^{-2}+S+S^{-1}$ and $\wt M_{\al}$ are unitarily equivalent.
Namely,  $\wt M_{\al}=Q\left(T^2+T^{-2}+S+S^{-1}\right)Q^{-1}$, where
\[
Q=U_1RU_{1/2}
\]
in terms of the operators $U_x$, $R$ given in (\ref{Q}), (\ref{R}).
\end{theorem}

\begin{proof}
First, we need commutation relations between the operators
(\ref{SandT})--(\ref{R}). We have
\[\begin{aligned}
(RS\phi)(n,\th)&=\sum_{k\in\mathbb{Z}}e^{-2\pi i
  k(\th+n\al)}\int_{\T}e^{-2\pi i n\bt}e^{2\pi
  i(\bt+k\al)}\phi(k,\bt)d\bt\\ &=
\sum_{k\in\mathbb{Z}}e^{-2\pi i
  k(\th+(n-1)\al)}\int_{\T}e^{-2\pi i
  (n-1)\bt}\phi(k,\bt)d\bt=(T^{-1}R\phi)(n,\th),
\end{aligned}\]
so that
\be\la{RS}
RS=T^{-1}R.
\ee
By taking the inverses
and then multiplying by $R$ from both sides we obtain
\be\la{RS-1}
RS^{-1}=TR.
\ee

Similarly to (\ref{RS}), (\ref{RS-1}), we obtain
\be\la{RTpm1}
RT=SR,\qquad RT^{-1}=S^{-1}R.
\ee

Furthermore, we have
\[\begin{aligned}
(TU_x\phi)(n,\th)&=
e^{2\pi i (n+1)x (\th+(n+1)\al/2)}\phi(n+1,\th)\\ &=
e^{i\pi x\al}e^{2\pi i x (\th+n\al)}e^{2\pi i n x (\th+n\al/2)}\phi(n+1,\th)=
(e^{i\pi x\al}S^xU_x T\phi)(n,\th),
\end{aligned}
\]
and so
\be\la{TU}
TU_x=e^{i\pi x\al}S^x U_xT,
\ee
from which it follows that
\be\la{UT}
U_xT=e^{-i\pi x\al}S^{-x}TU_x,\qquad U_xT^{-1}=e^{i\pi x\al}T^{-1}S^{x}U_x,
\ee
where the second relation is obtained by applying $T^{-1}$ from both
sides of (\ref{TU}).

Finally, it is easy to verify that
\be
S^x T= e^{-2\pi i x\al} TS^x,\qquad T^{-1}S^{-x}= e^{2\pi i
  x\al}S^{-x}T^{-1},\qquad U_xS^{y}=S^{y}U_x.
\ee

Using the above commutation relations we obtain
that, with $Q=U_1RU_{1/2}$,
\be\la{QS}
QS=U_1RS U_{1/2}=U_1T^{-1}RU_{1/2}=e^{i\pi\al}T^{-1}SU_1RU_{1/2}=e^{-i\pi\al}ST^{-1}Q,
\ee
i.e., $QS=e^{-i\pi\al}ST^{-1}Q$,
which upon taking the inverses, then applying $Q$ from both sides, and
then using the commutation relation for $TS^{-1}$ yields
\be\la{QS-1}
QS^{-1}=e^{-i\pi\al}S^{-1}TQ.
\ee
Similarly to (\ref{QS}), (\ref{QS-1}), we obtain
\be
QT^2=e^{i\pi\al}STQ,\qquad  QT^{-2}=e^{i\pi\al}S^{-1}T^{-1}Q.
\ee
Collecting the last 4 relations together, we obtain
\be
Q(T^2+T^{-2}+S+S^{-1})=(e^{i\pi\al}[ST+S^{-1}T^{-1}]+e^{-i\pi\al}[ST^{-1}+S^{-1}T])Q
=\wt M_{\al}Q,
\ee
and the statement of the theorem follows.
\end{proof}

\section{Proof of Theorem \ref{zero}}\label{proofzero}

For any real $\al$,
let $N_\al(E),\rho_\al$ and $\wt N_\al(E),\wt \rho_\al$ denote the
integrated densities of states and density of states measures of $H_{\al,\th}$ and $\wt H_{\al,\th}$, respectively,
and $L_\al(E)$ and $\wt L_\al(E)$ denote the corresponding Lyapunov exponents. We then have

\begin{lemma} \label{dn} $N_{2\al}=\wt N_\al.$\end{lemma}
\begin{proof} By Theorem \ref{chiralrepresentthm}, $T^2+T^{-2}+S+S^{-1}=
U\wt M_\al U^{-1}$ for a
unitary $U$ such that $U\Delta_0=\Delta_0$, where { $\Delta_0\in L^2(\T;\ell^2(\mathbb Z))$} is given by $\Delta_0(n,\th)=e_0(n)$.
Therefore, for any continuous function $\eta$
we have
\begin{align}\label{a}
 \int_{\mathbb{R}}
&\eta(E) d\wt N_\al(E)=\int_{\T}(\eta(\wt H_{\al,\th})e_0,e_0)\ d\theta=(\eta(\wt M_\al)\Delta_0,\Delta_0)\\
=&(\eta(U^{-1}[T^2+T^{-2}+S+S^{-1}]U)\Delta_0,\Delta_0)=
\int_{\T}(\eta( H^{(1)}_{2\al,\th}\oplus
H^{(2)}_{2\al,\th+\al})e_0,e_0)\ d\theta\notag\\
=&\int_{\T}(\eta(H^{(1)}_{2\al,\th})P_1^{-1}e_0,P_1^{-1}e_0)\
d\theta=\int_{\T}(\eta(P_1^{-1}H_{2\al,\th}P_1)P_1^{-1}e_0,P_1^{-1}e_0)\
d\theta\\ &
= \int_{\T}(\eta(H_{2\al,\th})e_0,e_0)\
d\theta
=\int_{\mathbb{R}}
\eta(E) d N_{2\al}(E),\notag
\end{align}
and the result follows. Similar
  proofs have been used in \cite{gjls,hj}. This also can be deduced from Theorem 2 of \cite{mz}.
\end{proof}

Let $\al$ be irrational. Note that $\si(M_{2\al})=\supp\, \rho_{2\al}=\supp\, \wt \rho_{\al}=\si(\wt M_\al).$
Since for $E\in \si(M_{2\al})$ we have $L_{2\al}(E)=0$ \cite{bj},
by (\ref{thou}), (\ref{thou2}), and Lemma \ref{dn} we also have $\wt
L_\al(E)=0$  for $E\in\si(\wt M_\al).$
The rest of the argument is the same as in Sec. 6.1 of \cite{bhj}. Namely, since the Jacobi matrix defined by
(\ref{chi}) is singular, the absolutely continuous spectrum is
empty \cite{do}. However, by Kotani theory,\footnote{The Kotani theory for non-singular Jacobi
  matrices is Theorem 5.17 in \cite{tes}. For singular
  matrices with $\ln |b(\th)| \in L^1$, it is Theorem 8 in \cite{bhj}.}
  it would not be empty if $|\{E\in\si(\wt M_\al): \wt L_\al(E)=0\}|>0$.
Thus $|\si(M_{\al})|=|\si(\wt M_{\al/2})|=0$. $\Box$

\section{Continuity of the spectrum}\label{sectioncontinuity}

Here we present the H\"older-$1/2$ continuity of the spectrum of \eqref{H2} for irrational $\al$ and
its improvement in the singular case,
which allows us to prove Theorem \ref{measure}, which, in turn, gives another proof of Theorem \ref{zero}.

We start with the following auxiliary lemma.
\begin{lemma}\label{lemma-S}{
Let $b(x)\in C^1(\mathbb R)$ be periodic of period $1$ and such that $b(0)=0$. Let $\al$ be an irrational number of bounded type.
Then there exists $C(b,\al)>0$ such that\footnote{We set $1/0=\infty$.}
for any $\th\in\bbr$ and any integer $L>1$,
\be\label{bt-ineq}
\sum_{k=0}^{L-1}\frac{1}{|b(\al k+\th)|}\ge C L\ln L.
\ee}
\end{lemma}

\begin{proof}{
Let $M=\|b'\|_\infty$. If $M=0$, then $b\equiv 0.$  Otherwise, for $q\ge4$, set
$g_q(x)=\min\{q,\|x\|_{\T}^{-1}\}$,
$g_q(0)=q$.
Then
$\int_{\T}g_q(x)\,dx=2+2\ln(q/2)$
and the variation
$\operatorname{Var}(g_q)=2q-4.$

Hence the Denjoy--Koksma inequality
 applied at  $q_n$  gives, uniformly in $\th$,
\[
\begin{aligned}
\sum_{k=0}^{q_n-1}g_{q_n}(\th+k\al)
\ge q_n\int_{\T}g_{q_n}(x)\,dx
      -\operatorname{Var}(g_{q_n})
=2q_n\ln(q_n/2)+4
 \ge q_n\ln q_n .
\end{aligned}
\]
Since $b(0)=0$ and $b$ is periodic,
$
|b(x)|\le M\|x\|_{\T},
$
and therefore
$$
\sum_{k=0}^{q_n-1}\frac1{|b(\th+k\al)|}
\ge \frac1M\sum_{k=0}^{q_n-1}g_{q_n}(\th+k\al)
\ge \frac{q_n\ln q_n}{M}.
$$
Let $L$ be sufficiently large and
choose $n$ so that $q_n\le L<q_{n+1}$. Since $\al$ is of bounded
type, $q_{n+1}\le A q_n$ for some $A=A(\al)$, and hence
$q_n\ge L/A$. This implies \eqref{bt-ineq}.
Since $b(x)$ is bounded, we now adjust the value of $C$ and obtain \eqref{bt-ineq} for any $L>1$.
}
\end{proof}

{ For any irrational $\alpha$ and $b$ as in Lemma~\ref{lemma-S}, set
$$
S_*(L):=\inf_{\theta\in\mathbb T}
\sum_{j=0}^{L-1}\frac{1}{|b(\theta+j\alpha)|}.
$$
The first part of the preceding proof gives
$S_*(q_k)\ge c q_k\ln q_k$ for all sufficiently large $k$,
with no further assumption on $\alpha$. For each fixed sufficiently large $k$, splitting a sum of length
$L$ into consecutive blocks of length $q_k$ gives
$
S_*(L)\ge \left\lfloor\frac{L}{q_k}\right\rfloor S_*(q_k),
$
and hence
\[
\liminf_{L\to\infty}\frac{S_*(L)}{L}
\ge \frac{S_*(q_k)}{q_k}\ge c\ln q_k.
\]
Letting $k\to\infty$ gives
\begin{equation}\label{sum-superlinear}
\lim_{L\to\infty}\frac{S_*(L)}{L}=\infty.
\end{equation}}

\begin{theorem}\label{continuitylemma1}
Let $H_{v,b,\alpha,\theta}$ and $M_{v,b,\alpha}$ be given by (\ref{H2}) and \eqref{M}, respectively,
with $v(x)$, $b(x)$ as described following (\ref{H2}).
Let $\al$ be irrational and $p_n/q_n$, $n=1,2,\dots$, its canonical {approximants}.
Then { there exists $C=C(\al,v,b)>0$ such that}
for any $E\in \si(M_{v,b,\al})$ there is $E'\in \si(M_{v,b,p_n/q_n})$ with
\be\label{cgeneral}
|E-E'| \le C \left|\al-\frac{p_n}{q_n}\right|^{1/2},\qquad n=1,2,\dots,
\ee
Moreover, { if $b$ has a zero on a period, then we also have
\begin{enumerate} \item if $\al$ is of bounded type,
 \be\label{cbounded}
 |E-E'| \le C \left|\al-\frac{p_n}{q_n}\right|^{1/2}\left|\ln\left|\al-\frac{p_n}{q_n}\right|\right|^{-1/2},\qquad n=1,2,\dots,
\ee \item for every irrational $\alpha$

\begin{equation}\label{csingular}
r_n:=
\sup_{E\in\sigma(M_{v,b,\alpha})}
\operatorname{dist}\bigl(E,\sigma(M_{v,b,p_n/q_n})\bigr)
=
o\left(\left|\alpha-\frac{p_n}{q_n}\right|^{1/2}\right),
\qquad n\to\infty.
\end{equation}
\end{enumerate}}
\end{theorem}

\begin{proof}
We modify the Schr\"odinger case argument
of \ci{ams} for the situation of $b\neq 1$.

Let $\al$ be irrational and $E\in \si(M_{v,b,\al})$. Fix a phase $\th\in\bbr$ such that $b(\al k+\th)\neq 0$, $k\in\bbz$.
For any $\ep>0$ we can find a { non-zero} $\phi_\ep\in\ell^2(\bbz)$ such that
 {$
\| (H_{v,b,\al,\th}- E)\phi_{\ep} \|\le \ep \|\phi_\ep\|.
$}
Denote $b_k=b(\al k+\th)$. { Let $L>2$ be an integer}, and, for any $m\in\bbz$, consider the sums
\be
S_{+,m}=\sum_{k=0}^{L-1}\frac{1}{|b_{k+m}|}=\sum_{k=m}^{m+L-1}\frac{1}{|b_k|},\qquad
S_{-,m}=\sum_{k=1}^{L}\frac{1}{|b_{-k+m}|}=\sum_{k=m-L}^{m-1}\frac{1}{|b_k|}.
\ee
Define the test function $f_{m,L}(n)$, $n\in\bbz$, centered at a site $m$, by the formulae:
\be
f_{m,L}(n)=
\begin{cases}
1,& n=m\\
1-\frac{1}{S_{+,m}}\sum_{k=m}^{n-1}\frac{1}{|b_k|},& n=m+1,\dots, m+L,\\
1-\frac{1}{S_{-,m}}\sum_{k=n}^{m-1}\frac{1}{|b_k|},& n=m-1,\dots, m-L,\\
0,&\text{otherwise}.
\end{cases}
\ee
Note that by this construction, $f_{m,L}(n)=0$ for $n\ge m+L$ and $n\le m-L$. Denote by the same symbol the corresponding diagonal operator $(f_{m,L}\phi)(n)=f_{m,L}(n)\phi(n)$, $n\in\bbz$.

\begin{lemma}\label{lemma-commut0}
The commutator $[f,H_{v,b,\al,\th}]$ has the following matrix elements
\be\label{commut0}
[f_{m,L},H_{v,b,\al,\th}]_{j,j+1}=-[f_{m,L},H_{v,b,\al,\th}]_{j+1,j}=
\begin{cases}
\frac{\sgn\, b_j}{S_{+,m}},& j=m,\dots m+L-1,\\
-\frac{\sgn\, b_j}{S_{-,m}},& j=m-1,\dots m-L,
\end{cases}
\ee
and all the other matrix elements are zero.

For any fixed $n\in\bbz$,
\be\label{sumest}
\frac{L+1}{2}\le\sum_m (f_{m,L}(n))^2\le 2L+1.
\ee
\end{lemma}

\begin{proof}
First note that since $[f_{m,L},H_{v,b,\al,\th}]^T=H_{v,b,\al,\th}^T f_{m,L}^T - f_{m,L}^T H_{v,b,\al,\th}^T=-[f_{m,L},H_{v,b,\al,\th}]$, the matrix of the commutator is skew-symmetric, and so it is sufficient to compute the elements at and above the main diagonal. We obtain by the definition of the test function that
\be
[f_{m,L},H_{v,b,\al,\th}]_{j,j+1}= (f_{m,L}(j)-f_{m,L}(j+1))b_j= \frac{\sgn\, b_j}{S_{+,m}},\qquad j=m,\dots m+L-1,
\ee
and similarly the other matrix elements as stated in the lemma.

{ The upper bound in \eqref{sumest} follows immediately from the definition of $f_{m,L}$.} { For the lower bound, fix $n\in\mathbb Z$
and, for $0\leq j\leq L$, set
$$
u_j=
\frac{\displaystyle\sum_{k=n}^{n-j+L-1}\frac{1}{|b_k|}}
{\displaystyle\sum_{k=n-j}^{n-j+L-1}\frac{1}{|b_k|}},
$$
where the numerator for $j=L$ is understood to be an empty sum.
Then $u_0=1$, $u_L=0$, and directly from the definition of
$f_{m,L}$,
{ $$\begin{aligned}
\sum_m \bigl(f_{m,L}(n)\bigr)^2
 &=1+
\sum_{m=n-L}^{n-1}\left(1-\frac{1}{S_{+,m}}\sum_{k=m}^{n-1}\frac{1}{|b_k|}\right)^2+
\sum_{m=n+1}^{n+L} \left(1-\frac{1}{S_{-,m}}\sum_{k=n}^{m-1}\frac{1}{|b_k|}\right)^2\\
 &=1+\sum_{j=1}^{L}u_j^2
 +\sum_{j=0}^{L-1}(1-u_j)^2
=1+\sum_{j=1}^{L-1}\bigl(u_j^2+(1-u_j)^2\bigr)
\geq
1+\frac{L-1}{2}
=\frac{L+1}{2},
\end{aligned}$$
}
where we used $x^2+(1-x)^2\geq 1/2$. This proves
\eqref{sumest}.}
\end{proof}

We will now show that $\| (H_{v,b,\al,\th} -E) f_{m,L}\phi_\ep \|$ is small for some $m$.
By the identity $\| u+v\|^2 \le (1+\de^{-1})\| u\|^2+ (1+\de)\| v\|^2$ with $\de=1$, we have
\be\label{c2}
\sum_m \| (H_{v,b,\al,\th} -E) f_{m,L}\phi_\ep \|^2\le
2\sum_m \| f_{m,L}(H_{v,b,\al,\th} -E) \phi_\ep\|^2+
2\sum_m\| [f_{m,L},H_{v,b,\al,\th}]\phi_\ep\|^2.
\ee
For the first sum on the r.h.s. we obtain by \eqref{sumest}
\[\begin{aligned}
&\sum_m \| f_{m,L}(H_{v,b,\al,\th} -E) \phi_\ep\|^2\le
(2L+1) \|(H_{v,b,\al,\th} -E) \phi_\ep\|^2\le (2L+1)\ep^2 \|\phi_\ep\|^2\\
&\le \frac{2L+1}{(L+1)/2}\ep^2 \sum_m \| f_{m,L}\phi_\ep\|^2\le 4\ep^2 \sum_m \| f_{m,L}\phi_\ep\|^2.
\end{aligned}
\]
For the second sum in \eqref{c2}, we obtain using the expressions for the matrix elements of the commutator given in Lemma \ref{lemma-commut0},
\[
\sum_m\| [f_{m,L},H_{v,b,\al,\th}]\phi_\ep\|^2\le
\frac{8L}{S_\mathrm{min}^2}\|\phi_\ep\|^2\le \frac{8L}{S_\mathrm{min}^2}
\frac{1}{(L+1)/2} \sum_m \| f_{m,L}\phi_\ep\|^2\le \frac{16}{S_\mathrm{min}^2}\sum_m \| f_{m,L}\phi_\ep\|^2,
\]
where
\be\label{defS}
S_\mathrm{min}=\min \left( \inf_m S_{+,m},  \inf_m S_{-,m}\right).
\ee

Therefore \eqref{c2} becomes
\[
\sum_m \| (H_{v,b,\al,\th} -E) f_{m,L}\phi_\ep \|^2\le
8\ep^2 \sum_m \| f_{m,L}\phi_\ep\|^2+\frac{32}{S_\mathrm{min}^2}\sum_m \| f_{m,L}\phi_\ep\|^2.
\]
Hence, there exists some $m$ { with $f_{m,L}\phi_\ep\ne0$} such that
\be\label{Hphi}
\| (H_{v,b,\al,\th} -E) f_{m,L}\phi_\ep \|^2\le
\left(8\ep^2+\frac{32}{S_\mathrm{min}^2}\right) \| f_{m,L}\phi_\ep\|^2.
\ee

Choose $\wh\th$ such that
$
\al m+\th= \frac{p_n}{q_n}m+\wh\th.
$

{ By the spectral theorem}, there exists $E'$ in the spectrum of $H_{v,b,p_n/q_n,\wh\th}$ such that
{
\begin{align}
|E'-E|&\le ||(H_{v,b,p_n/q_n,\wh\th}-E)f_{m,L}\phi_\ep||/||f_{m,L}\phi_\ep||\nonumber\\
&\le
||(H_{v,b,p_n/q_n,\wh\th}-H_{v,b,\al,\th})f_{m,L}\phi_\ep||/||f_{m,L}\phi_\ep||\nonumber\\
&\quad+
||(H_{v,b,\al,\th}-E)f_{m,L}\phi_\ep||/||f_{m,L}\phi_\ep||\nonumber\\
&\le
C|\al-p_n/q_n| L +||(H_{v,b,\al,\th}-E)f_{m,L}\phi_\ep||/||f_{m,L}\phi_\ep||.\label{contin}
\end{align}}

By \eqref{Hphi}, we have
\be\label{cont-med}
|E'-E|\le C|\al-p_n/q_n| L +4\sqrt{2}/S_\mathrm{min} + 2\sqrt{2}\ep.
\ee

Now consider two cases:

\noindent
1) Let $b(x)$ have a zero on its period. We then assume without loss of generality that $b(0)=0$. Let, first, $\al$ be of bounded type. Then Lemma \ref{lemma-S} shows that $ S_\mathrm{min} \ge C L \ln L$, and therefore
$|E'-E|\le C|\al-p_n/q_n| L +C'/(L\ln L) + 2\sqrt{2}\ep$. Choosing $L=\lfloor (|\al-p_n/q_n| |\ln |\al-p_n/q_n||)^{-1/2}\rfloor$, we obtain \eqref{cbounded} since $\ep>0$ is arbitrary. Adjusting $C$, we observe that \eqref{cbounded} holds for any $n\ge 1$.

{ For arbitrary irrational $\alpha,$ since $S_{\min}\ge S_*(L)$,
letting $\varepsilon\downarrow0$ in \eqref{cont-med} gives
$$
r_n\le C\left(
\left|\alpha-\frac{p_n}{q_n}\right|L+\frac{1}{S_*(L)}
\right)
$$
for every  $L>2$. Fix $A>0$. By
\eqref{sum-superlinear}, we have $S_*(L)\ge AL$ for all
sufficiently large $L$. Choosing
$
L=\left\lfloor
\left(A\left|\alpha-\frac{p_n}{q_n}\right|\right)^{-1/2}
\right\rfloor
$
therefore yields
$$
\limsup_{n\to\infty}
\frac{r_n}{|\alpha-p_n/q_n|^{1/2}}
\le \frac{C}{\sqrt A}.
$$
Letting $A\to\infty$
proves \eqref{csingular}.}

 \noindent
 2) Let $b(x)$ be any period 1 bounded function with no or at most a finite number of zeros on a period,
 and $\al$ any irrational number.
 Then using the trivial estimate $ S_\mathrm{min} \ge C L $
 we obtain \eqref{cgeneral} instead.

\end{proof}

\section{Proof of Theorem \ref{measure} and another proof of Theorem \ref{zero}}\label{sectionproofmeasure}

The proof is a combination of the argument of Last \cite{L} and Theorem \ref{continuitylemma1}.

{ Let $\al\in(0,1)$ be irrational and let $p_n/q_n$ be its canonical approximants.
Consider \eqref{H2} in the case when $b(x)$ has a zero on its period.
 By Hausdorff continuity of the spectrum \cite{AS1983,El82},}
\be|\si(M_{v,b,\al})|\ge \limsup_{n\to\infty} |\si(M_{v,b,p_{n}/q_{n}})|.\ee
We now show that $|\si(M_{v,b,\al})|\le \liminf_{n\to\infty} |\si(M_{v,b,p_n/q_n})|$.

{ The set $\si(M_{v,b,p_n/q_n})$ is a union of $q_n$ intervals:
\[
 \si(M_{v,b,p_n/q_n})=\cup_{m=1}^{q_n}[E_1^{m},E_2^{m}].
\]}
{ Let $r_n$ be defined in \eqref{csingular}, and set
\begin{equation}\label{wm}
w_m=[E_1^m-r_n,E_2^m+r_n].
\end{equation}
Then
$$
\sigma(M_{v,b,\alpha})\subseteq
\bigcup_{m=1}^{q_n}w_m.
$$
By \eqref{csingular} and \eqref{q}, we have $q_nr_n\to0$ as $n\to\infty$.
Consequently,
$$
|\sigma(M_{v,b,\alpha})|
{ \le} \sum_{m=1}^{q_n}|w_m| \le|\sigma(M_{v,b,p_n/q_n})|+o(1).
$$
Thus
$$
|\sigma(M_{v,b,\alpha})|
\le \liminf_{n\to\infty}|\sigma(M_{v,b,p_n/q_n})|,
$$
which completes the proof of Theorem~\ref{measure}.}

Since for the almost Mathieu operator
$|\si(M_{2p_n/q_n})|=|\si(M_{0,2\sin 2\pi(\cdot),p_{n}/q_{n}})|\le C/q_{n}$, we also obtain a proof of Theorem \ref{zero}.
$\Box$

\section{Finite covers for the spectrum of the critical almost Mathieu operator}\label{section-cover}

As described in the introduction, we will use the fact that the matrix of \eqref{chi} for irrational $\al$ is a collection of blocks of dimensions $q_n$ and $q_{n-1}$ connected by $2\times 2$ matrices of norm of order $1/q_n$, and construct a finite cover of the spectrum of \eqref{chi} by
at most $q_n+q_{n-1}$ intervals of total length of order $1/q_n$ around eigenvalues of the blocks.

\subsection{Separation of blocks}\label{section-separation}
We start with a general statement which shows that the spectrum of \eqref{H2}, in particular,
can be submerged into the closed union of spectra of perturbed blocks.

{
Consider the direct sum Hilbert space $\mathscr H=\bigoplus_{k\in\Z}\C^{N_k}$
with $N_k\ge2$.  Denote an orthonormal basis of $\C^{N_k}$ by
\[
e_0^{(k)},e_1^{(k)},\dots,e_{N_k-1}^{(k)},
\]
and write
\[
u^{(k)}=(u_0^{(k)},u_1^{(k)},\dots,u_{N_k-1}^{(k)})^T.
\]
We will sometimes omit the upper index $k$ in this notation.

For each $k$, let $B_k=B_k^*\in \mathcal M_{N_k}(\C)$ with $
 \sup_{k\in\Z}\|B_k\|<\infty.$
Then $D=\bigoplus_{k\in\Z}B_k$, $B_k: \C^{N_k}\to \C^{N_k}$,
is a bounded self-adjoint operator on $\mathscr H$ with norm $\| D\|_{\mathscr H}= \sup_{k\in\Z}\|B_k\|$.

 For each $k$, let $\Gamma_k: \C^{N_k}\to \C^{2}$ be given by
$\Gamma_k u^{(k)}=(u_0^{(k)},u_{N_k-1}^{(k)})^T$. Then  $\Gamma_k^*: \C^2\to \C^{N_k}$ acts as $\Gamma_k^*(x,y)^T=
(x,0,\dots,0,y)^T$.
Denote $\Gamma=\bigoplus_{k\in\Z}\Gamma_k$, which acts on $\mathscr H$.
Set $\mathscr F=\bigoplus_{k\in\Z}\mathrm{span}\{ e^{(k-1)}_{N_{k-1}-1}, e^{(k)}_0\}$. Then $\Gamma: \mathscr H \to \mathscr F$.

For each $k$, let $F_k=F_k^*\in \mathcal M_{2}(\C)$ with $
 \sup_{k\in\Z}\|F_k\|<\infty.$
Then $F=\bigoplus_{k\in\Z}F_k$, $F_k: \mathrm{span}\{ e^{(k-1)}_{N_{k-1}-1}, e^{(k)}_0\}\to \mathrm{span}\{ e^{(k-1)}_{N_{k-1}-1}, e^{(k)}_0\}$,
is a bounded self-adjoint operator on $\mathscr F$ with $\| F\|_{\mathscr F}= \sup_{k\in\Z}\|F_k\|$.
}

Let $A=D+\Gamma^*F\Gamma$.
Note that the blocks of $D$  in $A$ are connected by the action of $F$. Consider the operator $W=\Gamma_k^*V\Gamma_k: \C^{N_k}\to \C^{N_k}$, where $V$ is a $2\times 2$ matrix. Note that $W$ has at most $4$ nonzero matrix elements, and they are located in the corners,
$W_{0,0}$, $W_{0,N_k-1}$,$W_{N_k-1,0}$,$W_{N_k-1,N_k-1}$.

\begin{lemma}\label{lemma-subm}
Let $A$, $F_k$, $B_k$, $\Gamma_k$ be defined as above, and let $\tau=\|F\|_{\mathscr F} > 0$.  Then
\begin{equation}\label{specAsubm}
 \si(A)\subset
 \overline{\bigcup_{k\in\Z}
 \ \bigcup_{\substack{V=V^*\in \mathcal M_2(\C)\\ \|V\|\le2\tau}}
 \si(B_k+\Gamma_k^*V\Gamma_k)}.
\end{equation}
\end{lemma}

\begin{proof}
Fix $E\in\si(A)$.  We separate two cases.

{ First suppose that $E\in\si(D)$. Then} the result follows from $\si(D)=\overline{\bigcup_{k\in\Z}\si(B_k)},$ i.e., by choosing $V\equiv 0$.

Now let $E\notin\si(D)$.
Then the bounded inverse $(D-E)^{-1}$ exists.  Define the boundary resolvent
$G(E)=\Gamma(D-E)^{-1}\Gamma^*:\mathscr F\to\mathscr F$.
Because $D$ and $\Gamma$ are block diagonal, $G(E)=\bigoplus_{k\in\Z}G_k(E)$,
where
\begin{equation}\label{Gk-def}
 G_k(E)=\Gamma_k(B_k-E)^{-1}\Gamma_k^*\in \mathcal M_2(\C),\qquad G_k(E):\mathrm{span}\{e^{(k)}_0,e^{(k)}_{N_{k}-1}\}\to
 \mathrm{span}\{e^{(k)}_0,e^{(k)}_{N_{k}-1}\}.
\end{equation}
The operator $G(E)$ is bounded and self-adjoint.

Factor the reconnected operator as
\begin{align}
 A-E
 &=D-E+\Gamma^*F\Gamma =\left[\Id+\Gamma^*F\Gamma(D-E)^{-1}\right](D-E).
 \label{A-E-factor}
\end{align}
Denote
$ \mathcal X=\Gamma^*F:\mathscr F\to\mathscr H$,
$\mathcal Y=\Gamma(D-E)^{-1}:\mathscr H\to\mathscr F$. Then
$\mathcal X\mathcal Y=\Gamma^*F\Gamma(D-E)^{-1}$
and
 $\mathcal Y\mathcal X=G(E)F$.
Thus $(\Id+\mathcal X\mathcal Y)
\bigl(\Id-\mathcal X(\Id+\mathcal Y\mathcal X)^{-1}\mathcal Y\bigr)=\Id$, and therefore we obtain a standard identity:
\be\label{ABBA}
 (\Id+\mathcal X\mathcal Y)^{-1}
 =\Id-\mathcal X(\Id+\mathcal Y\mathcal X)^{-1}\mathcal Y.
 \ee

 Suppose for contradiction that $\|G(E)\|_{\mathscr F}<\tau^{-1}$. Then $\|G(E)\|_{\mathscr F}\,\|F\|_{\mathscr F}<1$,
 and therefore $\Id+\mathcal Y\mathcal X=\Id+G(E)F$
 has a bounded inverse $(\Id+\mathcal Y\mathcal X)^{-1}$ given by the {Neumann series}.
 Then by \eqref{ABBA}, there is a bounded inverse $(\Id+\mathcal X\mathcal Y)^{-1}$.
Since by \eqref{A-E-factor}, $(A-E)^{-1}=(D-E)^{-1}(\Id+\mathcal X\mathcal Y)^{-1}$, we see that $A-E$ has a bounded inverse and therefore $ E\notin\si(A)$, which contradicts our assumption.
Thus, for $E\in\si(A)\setminus\si(D)$, we have
\begin{equation}\label{G-bound}
 \|G(E)\|_{\mathscr F}\ge \tau^{-1}.
\end{equation}
{ Since  $\|G(E)\|_{\mathscr F}=\sup_{k\in\Z}\|G_k(E)\|,$}
\eqref{G-bound} guarantees
that there is an index $k$ satisfying $\|G_k(E)\|>(2\tau)^{-1}.$
Because { $G_k(E)$ is Hermitian, there is} a unit eigenvector
$v\in\C^2$ with real eigenvalue $\mu$ such that
\begin{equation}\label{mu-v}
 G_k(E)v=\mu v,
 \qquad
 |\mu|=\|G_k(E)\|>(2\tau)^{-1}.
\end{equation}
Define a rank-one Hermitian matrix $V=-\mu^{-1}vv^*$
where $vv^*$ is the rank-one operator $z\mapsto (z,v) v$ on
$\C^2$.  Then $\|V\|=|\mu|^{-1}<2\tau.$
Let $w=(B_k-E)^{-1}\Gamma_k^*v.$
Since $N_k\ge2$, the map $\Gamma_k^*:\C^2\to\C^{N_k}$ is injective and thus $w \neq 0$.
Using \eqref{Gk-def} and \eqref{mu-v}, we obtain
\begin{align*}
 (B_k+\Gamma_k^*V\Gamma_k-E)w&=(B_k-E)w
   +\Gamma_k^*V\Gamma_k(B_k-E)^{-1}\Gamma_k^*v\\
 &\quad=\Gamma_k^*(\Id+VG_k(E))v
 =\Gamma_k^*(v+\mu Vv)=0.
\end{align*}
Thus $E\in\si(B_k+\Gamma_k^*V\Gamma_k),$
with $\|V\|<2\tau$.  This proves
\eqref{specAsubm}.
\end{proof}

Now consider specifically $H_{v,b,\al,\th}$ given by  \eqref{H2} and assume it is singular. It is convenient to extend the definition of $b$ in \eqref{H2}, and we do it now, and assume that $b(x)$ has period 1 or period 2. If $b(x)$ has period 1, assume that $b(0)=0$. If $b(x)$ has period 2, assume that $b(0)=b(1)=0$.
We now show that $H_{v,b,\al,\th}$ in place of $A$ satisfies Lemma \ref{lemma-subm} with $\tau\le C/q_n$ and $N_k, k \in \Z,$ taking only two values, $q_n$, $q_{n-1}$.

Let
$
\de_n=q_n\left(\alpha-\frac{p_n}{q_n}\right),\qquad n=2,3,\dots.
$
Recall that by \eqref{pdord}, $\de_n>0$ for $n$ even, and $\de_n<0$ for $n$ odd. Define the intervals $J_{n,1}=(0,\de_{n-1}]$,  $J_{n,2}=(\de_{n},0]$ for $n$ odd, $J_{n,1}=[\de_{n-1},0)$,  $J_{n,2}=[0,\de_n)$ for $n$ even. Let $\wt J_n=J_{n,1}\cup J_{n,2}$. By \eqref{q},
the length $|\wt J_n|<2/q_n$.

{ It is well known that} if $x\in \wt J_n$, the orbit $z_k=x+\alpha k\pmod 1$, $k=1,2,\dots$, makes the first return to $\wt J_n$ after either $q_n$ or $q_{n-1}$ steps. For convenience, we present a brief proof.

\begin{lemma}\label{lemma-rt}
Let { $n\ge 2$.}
If $x\in J_{n,1}$ then $z_{q_n}\in\wt J_n$ and $z_k\notin \wt J_n$, $1\le k\le q_n-1$.
If $x\in J_{n,2}$ then $z_{q_{n-1}}\in\wt J_n$ and $z_k\notin \wt J_n$, $1\le k\le q_{n-1}-1$.
Moreover, $\| \al k \|_\T\ge|\de_{n-1}|$, $1\le k\le q_n-1$.
\end{lemma}

\begin{proof}
{ Let $T x=x+\alpha\pmod 1$, and $J_{n,1}$, $J_{n,2}$ be as defined above.  The two-tower partition of the circle associated with consecutive convergents is
\[
 \T=\bigsqcup_{j=0}^{q_n-1}T^jJ_{n,1}
 \ \sqcup\!
 \bigsqcup_{j=0}^{q_{n-1}-1}T^jJ_{n,2}.
\]
This is the standard
dynamical partition for an irrational circle rotation; see
\cite[Chapter~I, Section~1, especially~(1.6)]{dms}, used in a similar context in e.g. \cite{jl3,jmavi, jz}.  Hence the first return time
to $\wt J_n$ is $q_n$ on $J_{n,1}$ and $q_{n-1}$ on $J_{n,2}$.
Finally, the best-approximation property of the convergents gives
$
 \|k\alpha\|_{\T}\ge \|q_{n-1}\alpha\|_{\T}
 =|\de_{n-1}|,\; 1\le k<q_n,
$
which proves the last assertion.

}
\end{proof}

{ Fix any $n\ge 3$.}
{ Pick any phase $\th=\th_0$ of $H_{v,b,\al,\th}$ and choose
$n_0\in\mathbb Z$ such that $x_0 :=\th_0+n_0\al$ satisfy
$x_0\pmod{1}\in(-|\de_n|,|\de_n|)\subset\wt J_n$.}
By Lemma \ref{lemma-rt}, the points $z_k=x_0+\al k\pmod 1$, $k=1,2,\dots,$ return to
$\wt J_n$ each time after $q_n$ or $q_{n-1}$ steps.
Similarly, the points $-\wh z_k=-(x_0-\al k)\pmod 1=-x_0+\al k\pmod 1$ return to $\wt J_n$ each time after $q_n$ or $q_{n-1}$ steps,
and therefore $\wh z_k$ return to $-\wt J_n$.

{Define $m_0=0$. For $k\ge0$, let $m_{k+1}$ be the smallest integer
$m>m_k$ such that
$x_0+\al m\pmod 1\in\wt J_n$.
For $k\le0$, let $m_{k-1}$ be the largest integer
$m<m_k$ such that
$x_0+\al m\pmod 1\in-\wt J_n$.}
 Denote
{ \be\label{Jn}
J_n=\overline{-\wt J_n\cup \wt J_n}=[-|\de_{n-1}|,|\de_{n-1}|].
\ee}
By \eqref{q},
\be\label{Jn-length}
|J_n|<2/q_n.
\ee

We will separate $H_{v,b,\al,\th}$ into blocks removing the matrix elements $b(x_k)$, where { $x_k:=\th_0+(n_0+m_k)\alpha$. Then $x_k\pmod{1}\in J_n$ and $|b(x_k)|\le C(b)|J_n|<2C(b)/q_n$}. The length of the $k$'th block between the removed bonds $b(x_k)$ and $b(x_{k+1})$ is $N_k=m_{k+1}-m_k$.
{ By Lemma~\ref{lemma-rt}, $N_k\in\{q_n,q_{n-1}\}$, $N_k\ge 2$.}

 Denote an orthonormal basis of $\C^{N}$ by $e_0,e_1,\dots,e_{N-1}$,
{ and write a vector in this basis as} $u=(u_0,u_1,\dots, u_{N-1})^T$.
 Let $\wh \Gamma_N: \C^{N}\to \C^{2}$ and the projection $P_N:  \C^{N}\to \C^{N}$ onto  the $\operatorname{span}\{e_0,e_{N-1}\}$,
 be given by
\be\label{bp}
\wh\Gamma_Nu=(u_0,u_{N-1})^T,\qquad
P_N=\wh\Gamma_N^*\wh\Gamma_N.
\ee

 Define a tridiagonal matrix $\wh B_N(x):\C^{N}\to\C^{N}$ using the matrix elements of { $H_{v,b,\al,\th}$} as
 \begin{equation}\label{BhN}
\begin{aligned}
\wh B_N(x)_{j,j}&=v(x+(j+1)\alpha), \qquad 0\le j\le N-1,\cr
 \wh B_N(x)_{j+1,j}&=\wh B_N(x)_{j,j+1}=b(x+(j+1)\alpha),
 \qquad 0\le j\le N-2.
 \end{aligned}
\end{equation}

Then the $k$'th block between the removed bonds $b(x_k)$ and $b(x_{k+1})$  is $\wh B_{N_k}(x_k)$, and we consider $D(n)=D=\bigoplus_{k\in\Z}\wh B_{N_k}(x_k)$
on $\mathscr H_n=\bigoplus_{k\in\Z}\C^{N_k}.$ Furthermore, as $\Gamma_k$ from Lemma \ref{lemma-subm} depends on $k$ only via $N_k$, we have
$\wh \Gamma_{N_k}=\Gamma_k$,
and $\Gamma=\bigoplus_{k\in\Z}\Gamma_k=\bigoplus_{k\in\Z}\wh\Gamma_{N_k}$. Clearly, under the natural unitary identification
$U_\mathrm{Id}:\ell^2(\Z)\to \mathscr H_n$, we have
\be
U_\mathrm{Id} H_{v,b,\al,\th} U^*_\mathrm{Id}=D(n)+ \Gamma^* F(n) \Gamma,
\ee
where $F(n)=F$ from Lemma \ref{lemma-subm} with
\[
F_k=\begin{pmatrix} 0 & b(x_k)\cr b(x_k) & 0 \end{pmatrix},
\]
written in the basis $\{ e^{(k-1)}_{N_{k-1}-1}, e^{(k)}_0\}$.

We have that $\| F(n)\|_{\mathscr F}=\sup_{k\in\Z}|b(x_k)|$.
{ Since each $x_k\pmod{1}$ lies within distance $|J_n|$ of a zero of $b$, the mean value theorem gives
$\sup_k|b(x_k)|\le C(b)|J_n|$.}
Therefore $\|F(n)\|_{\mathscr F}\le C(b) |J_n|< 2C(b)/q_n$.

As $\si(H_{v,b,\alpha,\th})$ is independent of $\th$ for $\al$ irrational, applying Lemma \ref{lemma-subm} we obtain

\begin{proposition}
\label{prop-blocks}
Let $H_{v,b,\alpha,\th}$ be given by \eqref{H2} with { $b,v\in C^1(\mathbb R;\mathbb R)$}, $v(x)$ periodic of period $1$,
$b(x)$ either periodic of period $1$ with $b(0)=0$, or of period 2 and $b(0)=b(1)=0$. Let $\al$ be irrational, $\th\in\bbr$, $\wh\Gamma_N$ given in \eqref{bp},
{ $\wh B_N(x)$ given in \eqref{BhN}, and $J_n$ in \eqref{Jn}. Then for $N'_1=q_n$ and $N'_2=q_{n-1}$, $n\ge 3$,}
\begin{align}
 \si(H_{v,b,\alpha,\th})\subset
 \overline{\bigcup_{\nu=1}^{2}
 { \ \bigcup_{x\in J_n}}
 \ \bigcup_{\substack{V=V^*\in \mathcal M_2(\C)\\\|V\|\le 2C(b) | J_n|}}
 {\si\left(\wh B_{N'_\nu}(x)+\wh\Gamma_{N'_\nu}^*V\wh\Gamma_{N'_\nu}\right)}}.
\end{align}
\end{proposition}

{
\begin{remark}
The same conclusion holds for complex-valued, $1$-periodic $b$.  In
\eqref{H2}, the coefficient of $\phi(n-1)$ is then read as
$\overline{b(\theta+(n-1)\alpha)}$, and in \eqref{BhN} the subdiagonal
entry is the complex conjugate of the superdiagonal entry.  The removed
bonds give the Hermitian matrices
$
 \begin{pmatrix}0&b(x_k)\\ \overline{b(x_k)}&0\end{pmatrix},
$
whose norm is $|b(x_k)|$, so the proof is unchanged.
\end{remark}
}

We will show that for the chiral representation of the almost Mathieu operator with $v(x)\equiv 0$, $b(x)=2\sin(\pi x)$,
the union of spectra of the block $B_{N'_\nu}(x)+\wh\Gamma_{N'_\nu}^*V\wh\Gamma_{N'_\nu}$
over $V=V^*\in \mathcal M_2(\C)$, $\|V\|\le 4\pi | J_n |< 8\pi/q_n$, and over $x\in J_n$ is covered by at most $N'_\nu$ intervals of total length no larger than
$C/q_n$.

\subsection{Variation of block eigenvalues}

We now proceed with the unitary evolution argument described in the introduction. This { will allow us to conclude} for the almost Mathieu operator that if the phase $x$ of a block changes in $J_n$, $|J_n|<2/q_n$ (see \eqref{Jn},\eqref{Jn-length}), then the variation of the eigenvalues of the block is small.

We start with a general statement.
\begin{lemma}\label{lemma-flow}
Let $J=[a,b]$ be a compact interval.  Let
$H:J\to \mathcal M_N(\C)$ be a $C^1$ Hermitian path, $K=-K^*\in \mathcal M_N(\C)$ be a constant matrix, and let
 \begin{equation*}
\mathcal E(x)=
 \frac{d}{dx}H(x)-[K,H(x)],\qquad x\in J.
\end{equation*}
Let $\lambda_1(x)\le\lambda_2(x)\le\cdots\le\lambda_N(x)$
be the ordered eigenvalues of $H(x)$, and let $\Var_J\lambda_j$ denote the variation of $\lambda_j$ over $J$, i.e.,
$\Var_J\lambda_j=\sup_{\mathcal P}
 \sum_{\ell=1}^{M}
 |\lambda_j(y_\ell)-\lambda_j(y_{\ell-1})|$,
where { the supremum} is taken over all finite partitions of $J$.

Then
\begin{equation}\label{flowbound}
 \sum_{j=1}^{N}\Var_J\lambda_j
 \le\int_a^b\|\mathcal E(x)\|_{S_1}\,dx,
\end{equation}
where $\|\cdot\|_{S_1}$ denotes the trace norm.
\end{lemma}

\begin{proof}
Let $U(x)=e^{-(x-a)K}$. Since $K^*=-K$, we have $U(x)^*=e^{-(x-a)K^*}=e^{(x-a)K}=U(x)^{-1}$, so the function $U(x)$ is a unitary path.
Set
\begin{equation}\label{H-tilde}
 \widetilde H(x)=U(x)H(x)U(x)^*.
\end{equation}
Unitary conjugation preserves ordered eigenvalues, so $\lambda_j(\widetilde H(x))=\lambda_j(x).$
By differentiating \eqref{H-tilde}, we obtain, since $K$ is constant,
\begin{align*}
 \frac{d}{dx}\widetilde H(x)
 &=-KUHU^*+U\left(\frac{d}{dx}H\right)U^*+UHU^*K
 =U\left(\frac{d}{dx}H-[K,H]\right)U^*
=U\mathcal E U^*.
\end{align*}
In particular, unitary invariance of the trace norm gives
$\|\frac{d}{dx}\widetilde H(x)\|_{S_1}=\|\mathcal E(x)\|_{S_1}.$
The Lidskii inequality for Hermitian matrices
(e.g., \cite[Section~9]{Bhatia}) states that
\begin{equation*}
 \sum_{j=1}^{N}
 |\lambda_j(\widetilde H(y))-\lambda_j(\widetilde H(x))|
 \le\|\widetilde H(y)-\widetilde H(x)\|_{S_1},
\end{equation*}
and thus
\begin{equation}\label{partitionbound}
 \sum_{j=1}^{N}|\lambda_j(y)-\lambda_j(x)|
 \le  \|\widetilde H(y)-\widetilde H(x)\|_{S_1}\le\int_x^y\|\frac{d}{dt}\widetilde H(t)\|_{S_1}\,dt=\int_x^y\|\mathcal E(t)\|_{S_1}\,dt.
\end{equation}
{ For $\varepsilon>0$, choose for each $j$ a finite partition of
$J$ whose corresponding variation sum is at least
$\operatorname{Var}_J\lambda_j-\varepsilon/N$, and let
$\mathcal P$ be a common refinement of these partitions. Since
refinement can only increase each variation sum, summing
\eqref{partitionbound} over the subintervals of $\mathcal P$ gives
$$
\sum_{j=1}^N\operatorname{Var}_J\lambda_j-\varepsilon
\leq \int_a^b\|\mathcal E(t)\|_{S_1}\,dt.
$$
Letting $\varepsilon\downarrow0$ proves \eqref{flowbound}.}
\end{proof}

Consider now the almost Mathieu operator \eqref{chi} with $\al/2$,
\be\label{chi2}
 (\wh H_{\al/2,\th/2}\phi)(k)=b(\al (k-1) +\th)\phi(k-1)+
 b(\al k +\th)\phi(k+1),\qquad b(x)=2\sin(\pi x).
\ee
{ Note that  $b(x)=2\sin(\pi x)$ has period 2 and $b(0)=b(1)=0$, so the argument of Section \ref{section-separation}
applies to it. The purpose of this choice is to ensure that $b(x)$ has only one zero, $x=0$, for $x\pmod 1$,
and thus, the off-diagonal elements are sufficiently bounded away from zero inside each block.
We now construct an appropriate $K$ for \eqref{chi2} such that the corresponding $\mathcal E(x)$ is of finite rank.}

{ As above, fix any $n\ge 4$ and any $\th$.} For this $\th$, find an appropriate $k_1$ such that
{ $x_0=\al k_1+\th$, with
$x_0\pmod 1\in (-|\de_n|,|\de_n|)\subset J_n$},
and consider the $N\times N$ Jacobi block $\wh B_N(x)$ given by \eqref{BhN} with $v(x)\equiv 0$,
$b(x)=2\sin(\pi x)$. Denote it $B_N(x)$:
\begin{equation}\label{BN}
\begin{aligned}
B_N(x)_{j,j}&=0, \qquad 0\le j\le N-1,\cr
 B_N(x)_{j+1,j}&=B_N(x)_{j,j+1}=2\sin\pi(x+(j+1)\alpha),
 \qquad 0\le j\le N-2.
 \end{aligned}
\end{equation}
{ Let $N\ge 3$ and}
\[
L=\left\lfloor (N-1)/2\right\rfloor=\begin{cases}\frac{N-1}{2},& N\,  \mathrm{odd},\cr
\frac{N-2}{2},& N\, \mathrm{even}.\end{cases}
\]
Denote
\be
 c_r=
 \frac{(-1)^{r-1}\pi}{2\sin(\pi \al r)},\qquad r=1,2,\dots, L,
\ee
 where the denominators are nonzero because $\alpha$ is irrational.

Define the real skew-symmetric Toeplitz matrix $K_N\in \mathcal M_N(\mathbb R)$
entrywise by
\begin{equation}\label{defKL}
 (K_N)_{jk}=
 \begin{cases}
  0 & \text{if } k=j\text{ or }|k-j| \text{ is odd, }\\
  c_\frac{k-j}{2},& \text{if } k-j \text{ is even and } k>j,\\
  -c_\frac{j-k}{2},& \text{if } j-k \text{ is even and } k<j,
 \end{cases}\qquad 0\le j,k\le N-1.
\end{equation}
Thus $K_N^T=-K_N$, and $K_N$ depends on $\alpha$ and $N$ but not on $x$.

\begin{lemma}\label{lemma-commut}
{ Assume $N\ge 3$ and let $L=\lfloor (N-1)/2\rfloor$.}
Let $B_N(x)$  be given by \eqref{BN}, and $K_N$ by \eqref{defKL}.
Then
\begin{equation}\label{commutK1}
 \frac{d}{dx}B_N(x)=[K_N,B_N(x)]
 +\mathcal E_{\mathrm{even}}(x)
 +\mathcal E_{\mathrm{edge}}(x).
\end{equation}
The two matrices  $\mathcal E_{\mathrm{even}}(x)$,  $\mathcal E_{\mathrm{edge}}(x)$ are real symmetric.
The matrix $\mathcal E_{\mathrm{even}}(x)$ vanishes if $N$ is odd.  If $N$ is even, its only
nonzero entries are the two corners
\begin{equation}\label{Eeven}
 (\mathcal E_{\mathrm{even}})_{0,N-1}
 =(\mathcal E_{\mathrm{even}})_{N-1,0}
 =2\pi(-1)^{N/2-1}\cos(\pi(\al N/2+x)).
\end{equation}
The matrix $\mathcal E_{\mathrm{edge}}(x)$ is for $0\le r\le L-1$,
\begin{align}
 (\mathcal E_{\mathrm{edge}})_{0,2r+1}
 &=(\mathcal E_{\mathrm{edge}})_{2r+1,0}
 =-c_{r+1}2\sin(\pi x),
 \label{Eedge1}\\
 (\mathcal E_{\mathrm{edge}})_{N-2r-2,N-1}
 &=(\mathcal E_{\mathrm{edge}})_{N-1,N-2r-2}
 =c_{r+1}2\sin(\pi (x+N\al)).
 \label{Eedge2}
\end{align}
All the remaining entries of $\mathcal E_{\mathrm{even}}$ and
$\mathcal E_{\mathrm{edge}}$ are zero.
\end{lemma}

\begin{proof}
We drop the index $N$ of $K$ and $B=B(x)$ in the proof for convenience.
Denote $b_j=b_j(x)=2\sin\pi(\al j+x)$. Observe that the only nonzero elements of $B=B^T$ are
\be\label{Belements}
B_{j-1,j}=B_{j,j-1}=b_j,\qquad j=1,\dots,N-1.
\ee

Since
$[K,B]^T
 =B^T K^T-K^T B^T
 =-BK+KB
 =[K,B]$, it is sufficient to verify \eqref{commutK1} for the matrix elements at and above the central diagonal only.

 First, for any $1\le j\le N-2$, $\ell\ge 0$,
 \begin{align}
 [K,B]_{j,j+\ell}&=\sum_k( K_{jk}B_{k,j+\ell}-B_{jk}K_{k,j+\ell})\\
 &=
 \sum_r(K_{j,j+2r}B_{j+2r,j+\ell}-B_{j,j+\ell+2r}K_{j+\ell+2r,j+\ell})\\
 &=\sum_r(\sgn(r)c_{|r|}B_{j+2r,j+\ell}-(-\sgn(r))c_{|r|}B_{j,j+\ell+2r}),\label{KBintermed}
 \end{align}
 where the summations are over $k,r\in\bbz$ and the matrix elements of $K$, $B$ are understood to be zero for indices outside the range $0,\dots,N-1$, and {$c_{|r|}=0$ whenever $|r|\notin\{1,\dots,L\}$.}
  As the only nonzero elements of $B$ are given in \eqref{Belements} we immediately see that
 $[K,B]_{j,j+\ell}=0$ if $\ell$ is even, whereas for $\ell$ odd we obtain by \eqref{KBintermed} and \eqref{Belements},
 \be
 [K,B]_{j,j+\ell}=\begin{cases}
 c_\frac{\ell+1}{2}(b_{j+\ell+1}-b_j) + c_\frac{\ell-1}{2}(b_{j+\ell}-b_{j+1})=0,& \ell\ge 3,\quad j+\ell<N-1\\
 c_1(b_{j+2}-b_j)=\frac{d}{dx}b_{j+1}=\frac{d}{dx}B_{j,j+1},& \ell=1,\quad j+\ell<N-1\\
 -c_\frac{\ell+1}{2}b_N,&\ell\ge 3,\quad  j+\ell=N-1\\
 \frac{d}{dx}B_{N-2,N-1}-c_1 b_N,& \ell=1,\quad  j+\ell=N-1,
 \end{cases}
 \ee
 $1\le j\le N-2$, $\ell\ge 0$,
 where we used trigonometric identities at the last step.

 For $j=0$, $\ell\ge 0$, we have similarly that $[K,B]_{0,\ell}=0$ if $\ell$ is even, while if $\ell$ is odd we have:
 \be
 [K,B]_{0,\ell}=
 \begin{cases}
\frac{d}{dx}B_{0,1}+c_1b_0,& \ell=1,\\
c_\frac{\ell+1}{2}b_0,& 1<\ell<N-1,\\
c_\frac{N-2}{2}(b_{N-1}-b_1),& N\text{ even, }\ell=N-1.
\end{cases}
\ee
Finally for $j=N-1$, we also easily obtain that $[K,B]_{N-1,N-1}=0$.

Bringing the above observations together and recalling the Hermiticity of $[K,B]$, we see that
\[
[K,B]=\frac{d}{dx}B -  \mathcal E_{\mathrm{edge}} -\mathcal E_{\mathrm{even}},
\]
with the matrices $\mathcal E_{\mathrm{edge}}$, $\mathcal E_{\mathrm{even}}$ as described in the lemma.
\end{proof}

To deal with the perturbation $\wh\Gamma_{N}^*V\wh\Gamma_{N}$ of $B_N$ below, we will use the following modification of the previous lemma.

\begin{lemma}\label{lemma-commut2}
Define the two paths $H_N^{\pm}(x)=B_N(x)\pm C_0 | J_n| P_N$, for '+' and '-', respectively,  where $P_N$ is the boundary projection given in \eqref{bp},
$B_N(x)$ is the same as in { Lemma \ref{lemma-commut}} and $C_0>0$ is an absolute constant. Denote
\begin{equation}\label{Epm}
 \mathcal E_N^{\pm}(x)
 =\mathcal E_{\mathrm{even}}(x)
  +\mathcal E_{\mathrm{edge}}(x)
  \mp C_0 | J_n| [K_N,P_N],
\end{equation}
where $\mathcal E_{\mathrm{even}}(x)$, $\mathcal E_{\mathrm{edge}}(x)$ are given in { Lemma \ref{lemma-commut}} and { $K_N$ is defined} in \eqref{defKL}.
Then
\begin{equation}\label{commutK2}
\frac{d}{dx}H_N^{\pm}(x)
 =[K_N,H_N^{\pm}(x)]+\mathcal E_N^{\pm}(x).
\end{equation}
Moreover,  the rank $\operatorname{rk}\mathcal E_N^{\pm}(x)\le4$.
\end{lemma}

\begin{proof}
Since $P_N$ is independent of $x$, we conclude that $\frac{d}{dx}H_N^{\pm}(x)=\frac{d}{dx}B_N(x).$
On the other hand,
\begin{equation}
 [K_N,H_N^{\pm}(x)]
 =[K_N,B_N(x)]\pm C_0 | J_n| [K_N,P_N].
\end{equation}
Equation \eqref{commutK1} then implies \eqref{commutK2} and \eqref{Epm}.

Now note that $\mathcal E_N^{\pm}(x)e_j$, $j=1,\dots,N-2$, are in the span of $\{e_0,e_{N-1}\}$. Thus
$\mathcal E_N^{\pm}(x)$ { maps into the span of} $\{e_0,e_{N-1}, \mathcal E_N^{\pm}(x)e_0,\mathcal E_N^{\pm}(x)e_{N-1}\}$,
and so has rank $\operatorname{rk}\mathcal E_N^{\pm}(x)\le4$.
\end{proof}

Now we apply Lemma \ref{lemma-flow} with $H=H_N^{\pm}(x)$, $J=J_n$, $K=K_N$, and $\mathcal E(x)=\mathcal E_N^{\pm}(x)$, to estimate the sum of variations of eigenvalues.

\begin{proposition}\label{prop-phase-variation}
{ Let $N\in \{ q_{n-1},q_n\}$, $n\ge 4$.}
Let $\lambda_1^{\pm}(x)\le\cdots\le\lambda_N^{\pm}(x)$ be the ordered
eigenvalues of $H_N^{\pm}(x)$ (respectively) given in Lemma \ref{lemma-commut2}, and $x\in J_n$.  Then for an absolute $C>0$,
\begin{equation}\label{phase-variation-main}
 \sum_{j=1}^{N}\Var_{J_n}\lambda_j^{\pm}
 \le C |J_n|
 \le\frac{2C}{q_n}.
\end{equation}
\end{proposition}

\begin{proof} To use \eqref{flowbound}
we estimate the three parts of $\mathcal E_N^{\pm}(x)$ in
\eqref{Epm}. It suffices to estimate the Hilbert-Schmidt norms $\|\cdot\|_{S_2}$,
since $\mathcal E_N^{\pm}(x)$ is finite rank $r\le 4$
and hence
 $\|\mathcal E_N^{\pm}(x) \|_{S_1}\le \sqrt r\,\| \mathcal E_N^{\pm}(x)\|_{S_2}$.
 Using the explicit form \eqref{Eeven}, \eqref{Eedge1}, \eqref{Eedge2} of the matrix elements of $\mathcal E_N^{\pm}(x)$, it is easily seen that
 \be\label{Enorms}
 \begin{aligned}
 \|\mathcal E_{\mathrm{even}}(x)\|_{S_2}^2 &\le2(2\pi)^2=8\pi^2,\cr
 \|\mathcal E_{\mathrm{edge}}(x)\|_{S_2}^2
 &\le 2\bigl(|2\sin(\pi x)|^2+|2\sin(\pi (x+N\al))|^2\bigr)
       \sum_{r=1}^{L}|c_r|^2,\cr
\|[K_N,P_N]\|_{S_2}^2
& \le4\sum_{r=1}^{L}|c_r|^2.
 \end{aligned}
 \ee

To estimate the sum $\sum_{r=1}^{L}|c_r|^2$ note first that $L=\lfloor (N-1)/2\rfloor\le \lfloor q_n/2\rfloor$ as
{ $N=q_n$ or $N=q_{n-1}$, $n\ge 4$.}
We have $\al k=(\alpha-p_n/q_n)k+p_nk/q_n$, and by solving the congruences
$p_n k=j\pmod{q_n}$, we have $\al k(j)\pmod{1}=(\alpha-p_n/q_n)k(j)+ j/q_n\pmod{1}$, $j=1,\dots,q_n-1$. By \eqref{q}, these points
are located in the intervals
\be\label{intervals}
\left(\frac{j}{q_n}-\frac{1}{q_{n+1}},\frac{j}{q_n}+\frac{1}{q_{n+1}}\right) \pmod 1,\qquad j=0,1,\dots,q_n-1.
\ee

Moreover, by Lemma \ref{lemma-rt} and \eqref{q},
\[
\| \al k \|_\T\ge|\de_{n-1}|>\frac{1}{q_n+q_{n-1}}>\frac{1}{2 q_n},\qquad  1\le k\le q_n-1.
\]
Since $\sin(\pi x)$ has only one zero $x=0$ $\pmod 1$, we therefore obtain
\be
\begin{aligned}
\sum_{k=1}^{L}|c_k|^2
 &=\frac{\pi^2}{4}
   \sum_{k=1}^{L}\frac1{|\sin(\pi \alpha k)|^2} \le
\frac{\pi^2}{4}\left(
   \sum_{j=1}^{\lfloor q_n/2\rfloor}\frac{2}{|\sin(\pi j/q_n)|^2}+\frac{2}{|\sin(\pi /(2q_n))|^2}\right)\cr
   &< \frac{\pi^2}{4}\left( \frac{8 q_n^2}{\pi^2}\sum_{j=1}^{\infty}\frac{1}{j^2}+ \frac{32 q_n^2}{\pi^2}\right)< Cq_n^2,
\end{aligned}
\ee
with some absolute $C>0$.

Since $\al q_n\pmod{1}=(\al-p_n/q_n)q_n\pmod{1}=\de_n\pmod{1}$, $N\in\{ q_n,q_{n-1}\}$,
$x\in J_n=[-|\de_{n-1}|,|\de_{n-1}|]$,
and $|\de_n|<|\de_{n-1}|$,
we have that
\[
x+\al N\pmod{1}\in  [-2|\de_{n-1}|,2|\de_{n-1}|]:=J_n',\qquad  |J_n'|=2|J_n|<4/q_n.
\]
Therefore
$|\sin(\pi x)|<C/q_n$, $|\sin(\pi (x+N\al))|<C/q_n$, and \eqref{Enorms} becomes
\be\label{Enorms2}
\|\mathcal E_{\mathrm{even}}(x)\|_{S_2}^2 \le C_1,\qquad
 \|\mathcal E_{\mathrm{edge}}(x)\|_{S_2}^2 \le (C_2/ q_n^2)(C_3 q_n^2)\le C_4{ ,}\qquad
\|[K_N,P_N]\|_{S_2}^2\le C_5 q_n^2,
\ee
uniformly for $x\in J_n$.
Now apply Lemma \ref{lemma-flow} with $H=H_N^{\pm}(x)$, $J=J_n$, $K=K_N$, and $\mathcal E(x)=\mathcal E_N^{\pm}(x)$.
We obtain using \eqref{Enorms2} that
\be\begin{aligned}
 \sum_{j=1}^{N}\Var_{ J_n}\lambda_j^{\pm}
 &\le\int_{ J_n}
       \|\mathcal E_N^{\pm}(x)\|_{S_1}dx
 \le 2 \int_{ J_n}
       \|\mathcal E_N^{\pm}(x)\|_{S_2}dx\\
 &\le 2 \int_{ J_n}\|\mathcal E_{\mathrm{even}}(x)\|_{S_2}dx+
 2 \int_{ J_n}\|\mathcal E_{\mathrm{edge}}(x)\|_{S_2}dx+
 2|J_n|^2 C_0\|[K_N,P_N]\|_{S_2}\\
 & \le 2|J_n|(\sqrt{C_1}+\sqrt{C_4})+2|J_n|^2C_0\sqrt{C_5}q_n
 \le C|J_n|<\frac{2C}{q_n}.
\end{aligned}\ee
This proves \eqref{phase-variation-main}.
\end{proof}

\subsection{Spectral cover}
\label{subsection-cover}

We now bring the above results together.
\begin{proposition}\label{prop-cover1}
Let $B_N(x)$ be given by \eqref{BN}, $\widehat{\Gamma}_N$ by \eqref{bp}
{ with $N=N'_1=q_n$ or $N=N'_2=q_{n-1}$, $n\ge 4$, and $x\in J_n$.}
Let $V$ be a $2\times 2$ matrix, $V=V^*$, $\| V \|\le C_0 | J_n |$, with an absolute $C_0>0$.
Then there exist $N'_1$ intervals $I_j^{(1)}$ and $N'_2$ intervals $I_j^{(2)}$ independent of $x\in J_n$ and
$V$, such that the spectrum
\be\label{subBN}
\si\left(B_{N'_\nu}(x)+\widehat{\Gamma}_{N'_\nu}^*V\widehat{\Gamma}_{N'_\nu}\right)
 \subset\bigcup_{j=1}^{{N'_\nu}}I_j^{(\nu)},\qquad \nu=1,2,\qquad x\in J_n{ ,}
\ee
and, with an absolute $C>0$,
\be\label{lengthBN}
 \sum_{j=1}^{N'_\nu}|I_j^{(\nu)}|\le \frac{C}{q_n},\qquad \nu=1,2.
\ee
\end{proposition}

\begin{remark}
The intervals $I_j^{(\nu)}$ are not necessarily disjoint.
\end{remark}

\begin{proof}
{ Fix $N\in\{ q_n, q_{n-1}\}$, $n\ge 4$.} We clearly have
\[
-C_0 | J_n | P_N\le
-\|V\|P_N
 \le\wh\Gamma_N^*V\wh\Gamma_N
 \le\|V\|P_N\le C_0 | J_n | P_N,
 \]
where the projection $P_N=\wh\Gamma_N^*\wh\Gamma_N$ is defined in \eqref{bp}.

Adding $B_N(x)$ gives, with the notation of Lemma \ref{lemma-commut2},
\begin{equation}\label{Bbound}
 H_N^-(x)
 \le B_N(x)+\wh\Gamma_N^*V\wh\Gamma_N
 \le H_N^+(x).
\end{equation}

Let $\mu_1(x,V)\le\cdots\le\mu_N(x,V)$
be the ordered eigenvalues of $B_N(x)+\wh\Gamma_N^*V\wh\Gamma_N$, while
 $\lambda_1^{\pm}(x)\le\cdots\le\lambda_N^{\pm}(x)$ { are} those of
$H_N^{\pm}(x)$, respectively.
The min--max principle
applied to \eqref{Bbound} gives
\begin{equation}\label{minmax}
 \lambda_j^-(x)
 \le\mu_j(x,V)
 \le\lambda_j^+(x),
 \qquad 1\le j\le N.
\end{equation}
We then define \be \wt I_j=
 \left[
 \min_{x\in J_n}\lambda_j^-(x),
 \max_{x\in J_n}\lambda_j^+(x)
 \right].\ee
By
\eqref{minmax}, for every $x\in J_n$ we have
$
 \si\left(B_N(x)+\wh\Gamma_N^*V\wh\Gamma_N\right)
 \subset\bigcup_{j=1}^{N}\wt I_j.
$
Specializing this for each of the cases $N=N'_1$, $N=N'_2$, we obtain \eqref{subBN}.

We now show that \eqref{lengthBN} follows from Proposition \ref{prop-phase-variation}.
Fix $x_*\in J_n$. Clearly,
\[
 \max_{x\in J_n}\lambda_j^+(x)
 \le\lambda_j^+(x_*)+\Var_{J_n}\lambda_j^+,\qquad
 \min_{x\in J_n}\lambda_j^-(x)
 \ge\lambda_j^-(x_*)-\Var_{J_n}\lambda_j^-.
\]
Subtracting and summing over $j$ gives
\be
 \sum_{j=1}^{N}|\wt I_j|
 \le\sum_{j=1}^{N}
 \bigl(\lambda_j^+(x_*)-\lambda_j^-(x_*)\bigr) +
 \sum_{j=1}^{N}\Var_{J_n}\lambda_j^+
 +\sum_{j=1}^{N}\Var_{J_n}\lambda_j^-.
\ee
The first sum on the r.h.s. is a trace difference:
\[
 \sum_{j=1}^{N}
 \bigl(\lambda_j^+(x_*)-\lambda_j^-(x_*)\bigr)
 =\Tr\bigl(H_N^+(x_*)-H_N^-(x_*)\bigr)
 =\Tr(2C_0 | J_n| P_N)
 =4C_0 |J_n|\le\frac{C}{q_n},
\]
as $\Tr(P_N)=2$, $N>2$.
On the other hand,
Proposition~\ref{prop-phase-variation} gives
\[
 \sum_{j=1}^{N}\Var_{J_n}\lambda_j^+
 +\sum_{j=1}^{N}\Var_{J_n}\lambda_j^-
 \le 2C|J_n|\le\frac{C_1}{q_n},
\]
and therefore $\sum_{j=1}^{N}|\wt I_j| \le \frac{C}{q_n}$. Specializing this for $N=N'_1$ and $N=N'_2$ we obtain
\eqref{lengthBN}.
\end{proof}

\begin{theorem}\label{thm-cover}
Let $H_{\al,\th}=H_{\al,\th,1}$ be the critical almost Mathieu operator \eqref{AM} with $\al$ irrational, $\th\in\bbr$.
{ Let $p_n/q_n$, $n\ge 4$, be canonical {approximants} to $\al$.}
Then there exist an independent of $\al$, $n$, $\th$ constant $C>0$ and $m\le q_n+q_{n-1}$  independent of $\th$ closed intervals
$I_j$, $j=1,\dots,m$, of total length $\sum_{j=1}^{m}| I_j| \le \frac{C}{q_n}$ such that the spectrum
$\si(H_{\al,\th})\subset \bigcup_{j=1}^{m} I_j$.
\end{theorem}

\begin{proof}
Recall that by the chiral transform, the operators \eqref{AM} with $\lambda=1$ and \eqref{chi} with $\al/2$ are isospectral,
$\si(H_{\al,\th})=\si(\wh H_{\al/2,x})$, and the spectra are independent of $x,\th\in\bbr$ as $\al$ is irrational.
{ Fix any $n\ge 4$.} The statement now follows by Proposition \ref{prop-cover1} and Proposition \ref{prop-blocks} with
$H_{v,b,\al,\th}=H_{0,2\sin(\pi \cdot),\al,\th}=\wh H_{\al/2,\th/2}$ since the number $m$ of  intervals is finite.
\end{proof}

\section{{ Proof of Theorem \ref{Hdimthm}}}\label{section-hausdorff}

{ Fix an irrational $\al.$}
By Theorem \ref{thm-cover}, the spectrum of the critical almost Mathieu operator is covered by at most $q_n+q_{n-1}$ intervals $I_j$ of total length less { than} $C/q_n$,
\[
\si(H_{\al,\th})\subset \bigcup_{j=1}^{m} I_j,\qquad \sum_{j=1}^{m}| I_j| \le \frac{C}{q_n},\qquad m\le q_n+q_{n-1}<2q_n.
\]

By the H\"older inequality, we obtain
\be \label{61}
\sum_{j=1}^{m}| I_j |^t\le m^{1-t}\left(\sum_{j=1}^{m}|I_j|\right)^t\le C q_n^{1-2t},
\ee
{which tends to zero as $n\to\infty$ for any $1/2<t<1$.
Therefore $\meas_t(\si(H_{\al,\th}))=0$ for any $1/2<t<1$.
Hence $\dimH(\si(H_{\al,\th}))\le 1/2$, which proves Theorem \ref{Hdimthm}.}
$\Box$

{
\section{Appendix: A model of graphene.}\label{sectiongraphene}
Here we discuss an application of the { argument} of spectral covers
to another quantum model mentioned in the introduction.
A unified direct-cover formulation and its finer dimension consequences are given in the companion paper \cite[Theorem~1.3 and Corollary~1.4]{BJK2}.
 A model of graphene as a quantum graph on the
regular hexagonal lattice was introduced in \cite{bhj} and studied  in the presence of a magnetic
field with a constant flux $\Phi$. { See the precise definition} of the
Hamiltonian in (3.7) in \cite{bhj}. We denote its spectrum by $\sigma^\Phi$. Upon  identification with the
interval $[0, 1]$, the differential operator $V_{\vec e}$ in
(3.7) in \cite{bhj}, acting on each edge $\vec e$ is then the maximal Schr\"odinger operator
{ $-\frac{d^2}{dx^2}+V(x)$} with domain being the Sobolev space $H^2$, where { $V$ is a real-valued Kato-Rellich potential in $L^2[0,1]$ symmetric with respect to $1/2$} (see (2.12) in \cite{bhj}). We obtain

\begin{theorem}\label{gr}
For any { symmetric real-valued Kato-Rellich potential
  $V\in L^2$},
 the Hausdorff dimension $\dim_H (\sigma^{\Phi})\leq 1/2$ for all
 irrational $\frac\Phi {2\pi}$.
\end{theorem}

\begin{remark}
This result was proved in \cite{bhj} for a topologically generic but
measure zero set of $\Phi$.
\end{remark}

\begin{proof}
We first discuss corresponding blocks for the model.
Let $\alpha=\Phi/(2\pi)\pmod 1$ and translate the phase in
\cite[(5.4)]{bhj} by $1/2$.  The resulting complex Jacobi coefficients
are $c(x)=1-e^{-2\pi i x},\qquad v(x)=-2\cos(2\pi x).$
Writing $c_m=c(x+m\alpha)$, let $G_N(x)$ be the Hermitian matrix with
indices $0\le j,k\le N-1$ and entries
$
 (G_N)_{jj}=v(x+(j+1)\alpha),\qquad
 (G_N)_{j,j+1}=c_{j+1},\qquad
 (G_N)_{j+1,j}=\overline{c_{j+1}}.
$
Set $\varkappa_0=0$ and, for $r\ge1$,
\[
 \varkappa_r=
 \frac{\pi(-1)^{r-1}e^{\pi i r\alpha}}{\sin(\pi r\alpha)},
 \qquad
 \varkappa_{-r}=-\overline{\varkappa_r},
 \qquad
 (K_N^{\mathrm g})_{jk}=\varkappa_{k-j}.
\]
Thus $(K_N^{\mathrm g})^*=-K_N^{\mathrm g}$.  A direct computation,
using
$
 \varkappa_r\bigl(e^{-2\pi i r\alpha}-1\bigr)
 =2\pi i(-1)^r,
$
gives
$
 G_N'(x)=[K_N^{\mathrm g},G_N(x)]+E_N(x),
$
where
\[
 \begin{aligned}
 (E_N)_{jk}={}&
 \mathbf 1_{\{k=0\}}c_0\varkappa_{-j-1}
 +\mathbf 1_{\{k=N-1\}}\overline{c_N}\varkappa_{N-j}\\
 &-\mathbf 1_{\{j=0\}}\overline{c_0}\varkappa_{k+1}
 -\mathbf 1_{\{j=N-1\}}c_N\varkappa_{k-N}.
 \end{aligned}
\]
The two terms at the opposite corner must be combined before estimating:
\[
 (E_N)_{0,N-1}
 =\varkappa_N(\overline{c_N}-\overline{c_0})
 =2\pi i(-1)^Ne^{2\pi i(x+N\alpha)}.
\]
Let $\tau=C_0|J_n|$ and
$G_N^\pm(x)=G_N(x)\pm\tau P_N$.  Then
\[
 (G_N^\pm)'=[K_N^{\mathrm g},G_N^\pm]+E_N^\pm,\qquad
 E_N^\pm=E_N\mp\tau[K_N^{\mathrm g},P_N].
\]
These errors are supported on the boundary rows and columns, and hence
\[
 \operatorname{rank}E_N^\pm\le4,\qquad
 \|E_N^\pm\|_{S_1}
 \le C\left[
 1+\bigl(|c_0|+|c_N|+\tau\bigr)
 \left(\sum_{r=1}^{N-1}|\varkappa_r|^2\right)^{1/2}
 \right].
\]
For $N\in\{q_n,q_{n-1}\}$, Lemma~\ref{lemma-rt}, applied also to
differences of orbit points, shows that
$0,\alpha,\ldots,(N-1)\alpha$ are $|\de_{n-1}|$-separated on $\T$.
Consequently each successive distance shell about $0$ contains at most
two such points.  Since
$|\sin(\pi t)|\ge2\|t\|_{\T}$ and
$|\de_{n-1}|>1/(q_n+q_{n-1})$, it follows that
\[
 \sum_{r=1}^{N-1}|\varkappa_r|^2
 \le C|\de_{n-1}|^{-2}\le Cq_n^2.
\]
For $x\in J_n$, the endpoint estimates give
$|c_0|+|c_N|+\tau\le C/q_n$.  Therefore
$\sup_{x\in J_n}\|E_N^\pm(x)\|_{S_1}\le C$, and Lemma~\ref{lemma-flow}
yields, for the ordered eigenvalues $\lambda_j^\pm(x)$ of
$G_N^\pm(x)$,
\[
 \sum_{j=1}^N\Var_{J_n}\lambda_j^\pm
 \le C|J_n|\le\frac{C}{q_n}.
\]

Let $\Sigma_\Phi$ denote the
phase-independent spectrum of the auxiliary Jacobi operator in
\cite[(5.4)]{bhj}.  Proposition~\ref{prop-blocks} and the graphene
commutator estimate above allow us to repeat the proof of
Proposition~\ref{prop-cover1}.  Thus $\Sigma_\Phi$ has a cover by at
most $q_n+q_{n-1}$ intervals of total length at most $C/q_n$.
An argument as in the proof of Theorem \ref{Hdimthm} gives
$\dim_H(\Sigma_\Phi)\le1/2$.  The spectral reductions in
\cite[(5.3), Lemma~4.1, and the proof of Lemma~8.3]{bhj} now give the
assertion: the square-root and inverse-discriminant branches are locally
Lipschitz away from their branch points, while the exceptional branch
values and the Dirichlet spectrum are countable.

\end{proof}

}

\section*{Acknowledgements}
The work of S.J. was partially
supported by NSF DMS-1401204, 1901462, 2052899, 2155211, and
Simons-896624.   She was also a
2020-21 Simons Fellow. The work of I.K. was partially supported by the Leverhulme Trust
research programme grant RPG-2018-260.
I.K. is grateful to Jean Downes and Ruedi Seiler for their hospitality at TU Berlin, where part of this work was written.

\end{document}